\documentclass[a4paper,reqno,10pt]{article}

\usepackage{epsf,graphicx,epsfig}
\usepackage{amscd}
\usepackage{amsmath,latexsym,amssymb,amsthm}
\usepackage[nospace,noadjust]{cite}
\usepackage{textcomp}
\usepackage{setspace,cite}
\usepackage{lscape,fancyhdr,fancybox}
\usepackage{stmaryrd}
\usepackage[all,cmtip]{xy}
\usepackage{tikz}
\usepackage{cancel}
\usetikzlibrary{shapes,arrows,decorations.markings}
\usepackage{graphicx}

\usepackage{color}
\usepackage{url}
\usepackage{enumerate}
\usepackage[mathscr]{euscript}

  \theoremstyle{plain}

\swapnumbers
    \newtheorem{thm}{Theorem}[section]
    \newtheorem{prop}[thm]{Proposition}

    \newtheorem{subsec}[thm]{}
\theoremstyle{definition}
    \newtheorem{defn}[thm]{Definition}
        \newtheorem{remark}[thm]{Remark}
    \newtheorem{exam}[thm]{Example}

\theoremstyle{remark}

\title{}
\author{}
\date{}
\usepackage{amssymb}

\usepackage{hyperref}
\hypersetup{
	colorlinks,
	citecolor=blue,
	filecolor=black,
	linkcolor=blue,
	urlcolor=black
}

\begin{document}

\title{Twisted Rota-Baxter operators and Reynolds operators on Lie algebras and NS-Lie algebras}

\author{Apurba Das}

\maketitle

\begin{center}
Department of Mathematics and Statistics,\\
Indian Institute of Technology, Kanpur 208016, Uttar Pradesh, India.\\
Email: apurbadas348@gmail.com
\end{center}



\begin{abstract}
In this paper, we introduce twisted Rota-Baxter operators on Lie algebras as an operator analogue of twisted $r$-matrices. We construct a suitable $L_\infty$-algebra whose Maurer-Cartan elements are given by twisted Rota-Baxter operators. This allows us to define cohomology of a twisted Rota-Baxter operator. This cohomology can be seen as the Chevalley-Eilenberg cohomology of a certain Lie algebra with coefficients in a suitable representation. We study deformations of twisted Rota-Baxter operators from cohomological points of view. Some applications are given to Reynolds operators and twisted $r$-matrices. Next, we introduce a new algebraic structure, called NS-Lie algebras, that is related to twisted Rota-Baxter operators in the same way pre-Lie algebras are related to Rota-Baxter operators.
We end this paper by considering twisted generalized complex structures on modules over Lie algebras.
\end{abstract}

\medskip


algebras.

\medskip


\thispagestyle{empty}


\vspace{0.2cm}

\section{Introduction}
Rota-Baxter operators and more generally $\mathcal{O}$-operators on Lie algebras was first appeared in a paper of B. A. Kupershmidt \cite{kuper} as an operator analogue of classical $r$-matrices and Poisson structures. However, Rota-Baxter operators have origin in the work of G.-C. Rota \cite{rota} and G. Baxter \cite{baxter} in fluctuation theory of probability and combinatorics. They have important applications in the algebraic aspects of the renormalization in quantum field theory \cite{conn}. Rota-Baxter operators on Lie algebras (resp. associative algebras) are also related to the splitting of algebras, specifically with pre-Lie algebras (resp. dendriform algebras). Recently, deformations of Rota-Baxter operators on Lie algebras has been carried out in \cite{tang,laza-new}, and has been extended to the case of associative algebras in \cite{das-rota}. 

\medskip

The notion of twisted Poisson structures was introduced by P. \v{S}evera and A. Weinstein \cite{sev-wein} who showed that twisted Poisson structures can be described by Dirac structures in a Courant algebroid. C. Klim\v{c}\'{i}k and T. Strobl \cite{klimcik-strol} also studied twisted Poisson structures from geometric points of view. Motivated by this, K. Uchino \cite{uchino} defined a notion of twisted Rota-Baxter operator in the context of associative algebras as an operator analogue twisted Poisson structures and find relations with NS-algebras of P. Leroux \cite{leroux}. 

\medskip

Our primary aim in this paper is to introduce and study twisted Rota-Baxter operators on Lie algebras. Let $H$ be a $2$-cocycle in the Chevally-Eilenberg complex of a Lie algebra $\mathfrak{g}$ with coefficients in a representation $M$. A $H$-twisted Rota-Baxter operator is a linear map $T : M \rightarrow \mathfrak{g}$ that satisfies the Rota-Baxter type identity modified by the $2$-cocycle $H$ (see Definition \ref{defn-h-tw}).  The inverse of a $1$-cochain $h : \mathfrak{g} \rightarrow M$ (when $h$ is invertible) is a $H$-twisted Rota-Baxter operator with $H = - \delta_{\mathrm{CE}} h$. Here $\delta_{\mathrm{CE}}$ is the Chevalley-Eilenberg differential. 
If $r \in \wedge^2 \mathfrak{g}$ is a twisted $r$-matrix, then $r^\sharp : \mathfrak{g}^* \rightarrow \mathfrak{g}$ is a twisted Rota-Baxter operator (see Example \ref{exam-t-p}).
We provide some other examples including the one that arises from a Nijenhuis operator on a Lie algebra. 
Particular classes of twisted Rota-Baxter operators are given by Reynolds operators, which were first introduced in the associative algebra context (see \cite{zhang-gao-guo} for the literature). Reynolds operators are shown to be closely related with derivations on the Lie algebra.

\medskip 

In \cite{tang} the authors construct a graded Lie algebra using Voronov's derived bracket \cite{voro}, whose Maurer-Cartan elements are Rota-Baxter operators on Lie algebras. We introduce a ternary bracket (associated to a $2$-cocycle $H$) on the underlying graded vector space which makes it an $L_\infty$-algebra. The Maurer-Cartan elements of this $L_\infty$-algebra are precisely $H$-twisted Rota-Baxter operators. This allows us to define cohomology of a $H$-twisted Rota-Baxter operator. On the other hand, given a $H$-twisted Rota-Baxter operator, we construct a Lie algebra structure on $M$ and a representation of it on the vector space $\mathfrak{g}$. The corresponding Chevalley-Eilenberg cohomology is shown to be isomorphic with the cohomology of $T$.

\medskip

The classical formal deformation theory of Gerstenhaber \cite{gers} has been extended to Rota-Baxter operators on Lie algebras in \cite{tang}. We apply the same approach to $H$-twisted Rota-Baxter operators. We show that the infinitesimal in a deformation of $T$ is a $1$-cocycle in the cohomology of $T$. Moreover, we define a notion of equivalence between two formal deformations of $T$. The infinitesimals corresponding to equivalent deformations are shown to be cohomologous. We introduce Nijenhuis elements associated with a $H$-twisted Rota-Baxter operator. Such Nijenhuis elements are obtained from trivial linear deformations. We also find a sufficient condition for the rigidity of a $H$-twisted Rota-Baxter operator in terms of Nijenhuis elements.


\medskip

In the next, we introduce a new algebraic structure consisting of two binary operations, the second of which is skew-symmetric and these two operations are supposed to satisfy two new identities. We call such algebras as NS-Lie algebras. They are a generalization of pre-Lie algebras and a skew-symmetric analogue of (associative) NS-algebras introduced by Leroux \cite{leroux}. An NS-Lie algebra naturally associates a Lie algebra structure, called the adjacent Lie algebra. We give some examples of NS-Lie algebras, specially obtained from Nijenhuis operators, (associative) NS-algebras and $H$-twisted Rota-Baxter operators. We show the category of twisted Rota-Baxter operators and the category of NS-Lie algebras are equivalent. More details study of NS-Lie algebras and their adjacent Lie algebras will be treated in a subsequent paper.

\medskip

Finally, inspired by the generalized complex structures \cite{gual,hitchin} on differential geometry, we introduce $H$-twisted generalized complex structures on modules over Lie algebras. Such a structure has an underlying Rota-Baxter operator. Motivated from the explicit description of a generalized complex structure, we obtain the explicit form of a twisted generalized complex structure on a module over a Lie algebra. Some examples are also given.

\medskip

The paper is organized as follows. In the next section (section \ref{sec-2}), we introduce $H$-twisted Rota-Baxter operators and Reynolds operators on Lie algebras. In section \ref{sec-mc}, we consider an $L_\infty$-algebra whose Maurer-Cartan elements are given by $H$-twisted Rota-Baxter operators. This motivates us to define cohomology of a $H$-twisted Rota-Baxter operator $T:M \rightarrow \mathfrak{g}$. We show that a $H$-twisted Rota-Baxter operator $T$ induces a Lie algebra structure on $M$ which has representation on $\mathfrak{g}$. The corresponding Chevalley-Eilenberg cohomology is shown to be isomorphic with the cohomology of $T$. Deformations of $H$-twisted Rota-Baxter operators from cohomological points of view has been treated in section \ref{sec-4}. Some applications to Reynolds operators and twisted $r$-matrices are given in section \ref{sec-app}. We introduce NS-Lie algebras and obtain their relation with Nijenhuis operators, (associative) NS-algebras and $H$-twisted Rota-Baxter operators in section \ref{sec-ns}. Finally, in section \ref{sec-tgcs}, we consider twisted generalized complex structures on modules over Lie algebras.

\medskip

All (graded) vector spaces, linear maps, tensor products and wedge products are over a field $\mathbb{K}$ of characteristic zero.

\section{Twisted Rota-Baxter operators and Reynolds operators on Lie algebras}\label{sec-2}

Let $(\mathfrak{g}, [~,~])$ be a Lie algebra. A representation of $\mathfrak{g}$ consists of a vector space $M$ together with a bilinear map $\bullet : \mathfrak{g} \times M \rightarrow M$, $(x, m) \mapsto x \bullet m$ satisfying the following
\begin{align*}
[x,y] \bullet u = x \bullet (y \bullet u ) - y \bullet ( x \bullet u),~ \text{ for } x, y \in \mathfrak{g}, u \in M.
\end{align*}

Given a Lie algebra $\mathfrak{g}$ and a representation $M$, the Chevalley-Eilenberg cohomology of $\mathfrak{g}$ with coefficients in $M$ is given by the cohomology of the cochain complex $(\{ C^n_{\mathrm{CE}}  (\mathfrak{g}, M) \}_{n \geq 0}, \delta_{\mathrm{CE}} )$, where $C^n_{\mathrm{CE}}(\mathfrak{g}, M) = \mathrm{Hom}(\wedge^n \mathfrak{g}, M )$, for $n \geq 0$ and the differential $\delta_{\mathrm{CE}} : C^n_{\mathrm{CE}}(\mathfrak{g}, M) \rightarrow C^{n+1}_{\mathrm{CE}}(\mathfrak{g}, M)$ given by
\begin{align*}
(\delta_{\mathrm{CE}} f)(x_1, \ldots, x_{n+1}) =~& \sum_{i=1}^{n+1} (-1)^{i+1}~ x_i \bullet f (x_1, \ldots, \widehat{x_i}, \ldots, x_{n+1})\\
& + \sum_{1 \leq i <j \leq n+1} (-1)^{i+j} f ([x_i, x_j ], x_1, \ldots, \widehat{x_i}, \ldots, \widehat{x_j}, \ldots, x_{n+1}).
\end{align*}  
The corresponding cohomology groups are denoted by $H^n_{\mathrm{CE}}(\mathfrak{g}, M)$, for $n \geq 0$.

Let $H \in C^2_{\mathrm{CE}}(\mathfrak{g}, M)$ be a $2$-cocycle, i.e. $H : \mathfrak{g} \times \mathfrak{g} \rightarrow M$ is a skew-symmetric bilinear map satisfing
\begin{align*}
x \bullet H (y, z) + y \bullet H (z,x) + z \bullet H (x, y) + H (x, [y,z]) + H (y, [z,x]) + H (z, [x,y]) = 0,~ \text{ for } x,y, z \in \mathfrak{g}.
\end{align*}

\begin{defn}\label{defn-h-tw}
A linear map $T: M \rightarrow \mathfrak{g}$ is said to a {\bf $H$-twisted Rota-Baxter operator} if $T$ satisfies
\begin{align}\label{iden-H-twisted}
[T(u), T(v)] = T ( T(u) \bullet v - T(v) \bullet u + H (Tu, Tv)),~ \text{ for } u, v \in M.
\end{align}
\end{defn}

Let $\mathfrak{g}$ be a Lie algebra and $M$ be a $\mathfrak{g}$-module. For any $2$-cocycle $H$, the direct sum $\mathfrak{g} \oplus M$ carries a Lie algebra structure given by
\begin{align}\label{semi-dir-brkt}
[(x,u), (y, v)]^H = ([x,y], x \bullet v - y \bullet u + H(x, y)).
\end{align}
This is called the $H$-twisted semi-direct product, denoted by $\mathfrak{g} \ltimes_H M$.

\begin{prop}\label{graph-twisted}
A linear map $T: M \rightarrow \mathfrak{g}$ is a $H$-twisted Rota-Baxter operator if and only if the graph $\mathrm{Gr}(T) = \{ (Tu, u) | ~ u \in M \}$ is a subalgebra of the $H$-twisted semi-direct product $\mathfrak{g} \ltimes_H M$.
\end{prop}

\begin{remark}
Rota-Baxter operators on Lie algebras generalize classical {\bf r}-matrices \cite{kuper}, hence generalize Poisson structures on a manifold. This can be viewed by the following similarities between Rota-Baxter operators and Poisson structures. A linear map $T : M \rightarrow \mathfrak{g}$ is a Rota-Baxter operator if and only if $\mathrm{Gr}(T) \subset \mathfrak{g} \oplus M$ is a subalgebra of the semi-direct product, on the other hand, a bivector field $\pi \in \Gamma (\wedge^2 TM)$ on a manifold is a Poisson tensor if and only if the graph $\mathrm{Gr}(\pi^\sharp)$ of the bundle map $\pi^\sharp : T^*M \rightarrow TM$ is closed under the Courant bracket $[~,~]_{\mathrm{Cou}}$ on $TM \oplus T^*M$. In \cite{sev-wein} \v{S}evera and Weinstein introduced a $H$-twisted Courant bracket $[~,~]_{\mathrm{Cou}}^H$ on $TM \oplus T^*M$, for any closed $3$-form $H \in \Omega^3_{\mathrm{cl}}(M)$. A bivector field $\pi$ is called a $H$-twisted Poisson structure if $\mathrm{Gr}(\pi^\sharp)$  is closed under the $H$-twisted Courant bracket $[~,~]_{\mathrm{Cou}}^H$. In this regard, the notion of $H$-twisted Rota-Baxter operators are generalization of $H$-twisted Poisson structures.
\end{remark}

The following result is a consequence of Proposition \ref{graph-twisted}.

\begin{prop}\label{graph-twisted-new}
Let $T : M \rightarrow \mathfrak{g}$ be a $H$-twisted Rota-Baxter operator. Then $M$ carries a Lie algebra structure with bracket
\begin{align*}
[u, v]_T := T(u) \bullet v - T(v) \bullet u + H (Tu, Tv),~ \text{ for } u, v \in M.
\end{align*}
\end{prop}

\begin{defn}\label{rota-morphism}
Let $T : M \rightarrow \mathfrak{g}$ be a $H$-twisted Rota-Baxter operator. Suppose $\mathfrak{g}'$ is another Lie algebra, $M'$ is a $\mathfrak{g}'$-module and $H' \in C^2_{\mathrm{CE}} (\mathfrak{g}', M')$ is a $2$-cocycle. Let $T' : M' \rightarrow \mathfrak{g}'$ be a $H'$-twisted Rota-Baxter operator. A morphism of twisted Rota-Baxter operators from $T$ to $T'$ consists of a pair $(\phi, \psi)$ of a Lie algebra morphism $\phi : \mathfrak{g} \rightarrow \mathfrak{g}'$ and a linear map $\psi : M \rightarrow M'$ satisfying
\begin{align*}
\psi ( x \bullet u ) = \phi (x) \bullet \psi (u), \quad \psi \circ H = H' \circ (\phi \otimes \phi) ~~~ \text{ and } ~~~ \phi \circ T = T' \circ \psi, ~ \text{ for } x \in \mathfrak{g}, u \in M.
\end{align*}
\end{defn}

\begin{exam}
Any Rota-Baxter operator of weight $0$ (more generally any $\mathcal{O}$-operator or Kuperschmidt operator) on a Lie algebra is a $H$-twisted Rota-Baxter operator with $H = 0$.
\end{exam}

\begin{exam}
Let $\mathfrak{g}$ be a Lie algebra and $M$ be a $\mathfrak{g}$-module. Suppose $h : \mathfrak{g} \rightarrow M$ is an invertible $1$-cochain in the Chevalley-Eilenberg cochain complex of $\mathfrak{g}$ with coefficients in $M$. Then $T = h^{-1} : M \rightarrow \mathfrak{g}$ is a $H$-twisted Rota-Baxter operator with $H = - \delta_{\mathrm{CE}} h$. To verify this, we observe that
\begin{align}\label{exam-1-co}
H(Tu, Tv) = - (\delta_{\mathrm{CE}} h )(Tu, Tv) = - T(u)\bullet v + T(v) \bullet u + h ([Tu,Tv]).
\end{align}
By applying $T$ to both sides of (\ref{exam-1-co}), we get the identity (\ref{iden-H-twisted}).
\end{exam}

\begin{exam}
Let $\mathfrak{g}$ be a Lie algebra with Lie bracket given by the map $\mu : \wedge^2 \mathfrak{g} \rightarrow \mathfrak{g}$. Note that the space $M = \wedge^2 \mathfrak{g}$ is a $\mathfrak{g}$-module by $[x, y \wedge z] = [x,y]\wedge z + y \wedge [x,z]$, for $x \in \mathfrak{g}$ and $y \wedge z \in \wedge^2 \mathfrak{g}$. Moreover, the map $H : \wedge^2 \mathfrak{g} \rightarrow \wedge^2 \mathfrak{g}$, $y \wedge z \mapsto - y \wedge z$ is a $2$-cocycle in the Chevalley-Eilenberg cohomology of $\mathfrak{g}$ with coefficients in $\wedge^2 \mathfrak{g}$. With this notation, the map $\mu : \wedge^2 \mathfrak{g} \rightarrow \mathfrak{g}$ is a $H$-twisted Rota-Baxter operator.
\end{exam}

\begin{exam}\label{exam-n-tw}
Let $N : \mathfrak{g} \rightarrow \mathfrak{g}$ be a Nijenhuis operator on a Lie algebra $\mathfrak{g}$, i.e. $N$ satisfies
\begin{align*}
[Nx, Ny] = N ( [Nx, y] + [x, Ny] - N [x,y]),~ \text{ for } x, y \in \mathfrak{g}.
\end{align*}
In this case $\mathfrak{g}$ carries a new Lie bracket $[x,y]_N = [Nx, y] + [x, Ny] - N[x,y]$, for $x, y \in \mathfrak{g}$. We denote this Lie algebra structure by $\mathfrak{g}_N$. Moreover, the Lie algebra $\mathfrak{g}_N$ has a representation on $\mathfrak{g}$ by $x \bullet y := [Nx, y]$, for $x \in \mathfrak{g}_N$, $y \in \mathfrak{g}$. With this representation, the map $H : \wedge^2 \mathfrak{g}_N \rightarrow \mathfrak{g}$, $H (x, y) = - N [x, y]$ is a $2$-cocycle in the Chevalley-Eilenberg cohomology of $\mathfrak{g}_N$ with coefficients in $\mathfrak{g}$. Then it is easy to observe that the identity map $\mathrm{id} : \mathfrak{g} \rightarrow \mathfrak{g}_N$ is a $H$-twisted Rota-Baxter operator.

This example will be more clear in Section  \ref{sec-ns} when we will introduce NS-Lie algebras and a functor from the category of NS-Lie algebras to the category of twisted Rota-Baxter operators.
\end{exam}

\begin{exam}\label{exam-t-p} (Twisted triangular $r$-matrix)
Let $\mathfrak{g}$ be a Lie algebra and $\psi \in \wedge^3 \mathfrak{g}^*$ be a $3$-cocycle of $\mathfrak{g}$ with coefficients in $\mathbb{K}$. Then an element ${\bf r} \in \wedge^2 \mathfrak{g}$ is called a $\psi$-twisted triangular $r$-matrix if ${\bf r}$ satisfies
\begin{align}\label{tw-eqn}
[{\bf r}_{12}, {\bf r}_{13}] + [{\bf r}_{12}, {\bf r}_{23}] + [{\bf r}_{13}, {\bf r}_{23}] = - (\wedge^3 {\bf r}^\sharp )(\psi),
\end{align}
where ${\bf r}^\sharp : \mathfrak{g}^* \rightarrow \mathfrak{g},~ \alpha \mapsto {\bf r}( \alpha, ~ )$ is the map induced by ${\bf r}$. The equation (\ref{tw-eqn}) is called the {\em $\psi$-twisted classical Yang-Baxter equation}. Note that twisted $r$-matrices are Lie algebraic version of twisted Poisson structures \cite{sev-wein}.

Observe that, $\psi \in \wedge^3 \mathfrak{g}^*$ is a $3$-cocycle in the cohomology of $\mathfrak{g}$ with coefficients in $\mathbb{K}$ imples that $\psi^\sharp : \wedge^2 \mathfrak{g} \rightarrow \mathfrak{g}^*$ is a $2$-cocycle in the cohomology of $\mathfrak{g}$ with coefficients in the coadjoint representation $\mathfrak{g}^*$. With this notation, an element ${\bf r} \in \wedge^2 \mathfrak{g}$ is a $\psi$-twisted triangular $r$-matrix if and only if ${\bf r}^\sharp : \mathfrak{g}^* \rightarrow \mathfrak{g}$ is a $\psi^\sharp$-twisted Rota-Baxter operator on $\mathfrak{g}^*$ over the Lie algebra $\mathfrak{g}$.

Let ${\bf r } \in \wedge^2 \mathfrak{g}$ be a $\psi$-twisted triangular $r$-matrix. Then it has been observed in \cite{yks-yaki} that $\mathfrak{g}^*$ carries a Lie algebra structure with bracket
\begin{align}\label{tw-bracket}
[\alpha, \beta]_{{\bf r}, \psi} := ad^*_{{\bf r}^\sharp (\alpha)} \beta - ad^*_{{\bf r}^\sharp (\beta)} \alpha + \psi ( {\bf r}^\sharp \alpha, {\bf r}^\sharp \beta, ~),~ \text{ for } \alpha, \beta \in \mathfrak{g}^*.
\end{align}
Moreover, the map ${\bf r}^\sharp : \mathfrak{g}^* \rightarrow \mathfrak{g}$ is a Lie algebra morphism, equivalently, 
$[{\bf r}^\sharp \alpha, {\bf r}^\sharp \beta ] = {\bf r}^\sharp ([\alpha, \beta ]_{{\bf r}, \psi})$, for $\alpha, \beta \in \mathfrak{g}^*$. 
\end{exam}

Given a $H$-twisted Rota-Baxter operator $T$ and a $1$-cochain $h$, here we construct a $(H + \delta h)$-twisted Rota-Baxter operator under certain condition. First we observe the following.

\begin{prop}\label{coboun}
Let $\mathfrak{g}$ be a Lie algebra and $M$ be a $\mathfrak{g}$-module. For any $2$-cocycle $H \in C^2_{\mathrm{CE}} (\mathfrak{g}, M)$ and $1$-cochain $h \in C^1_{\mathrm{CE}} ( \mathfrak{g}, M)$, we have isomorphism of Lie algebras
\begin{align*}
\mathfrak{g} \ltimes_H M ~\cong~ \mathfrak{g} \ltimes_{H + \delta h} M. 
\end{align*}
\end{prop}

\begin{proof}
Define $\Psi_h : \mathfrak{g} \ltimes_H M \rightarrow \mathfrak{g} \ltimes_{H + \delta h} M$ by $\Psi_h (x, u) = (x, u - h (x))$, for $(x, u) \in \mathfrak{g} \oplus M$. Then we have
\begin{align*}
\Psi_h [ (x, u), (y, v)]^H =~& ( [x, y] , x \bullet v - y \bullet u + H (x, y) - h [x, y]) \\
=~& ( [x, y], x \bullet v - y \bullet u + H (x, y) - x \bullet h (y) + y \bullet h(x) + (\delta h) (x, y)) \\
=~& [(x, u - h(x)), ( y , v - h(y))]^{H + \delta h} \\
=~& [ \Psi_h (x, u), \Psi_h (y, v)]^{H + \delta h}.
\end{align*}
This proves the result.
\end{proof}

Let $T : M \rightarrow \mathfrak{g}$ be a $H$-twisted Rota-Baxter operator and $h \in C^1_{\mathrm{CE}} (\mathfrak{g}, M)$
 be a $1$-cochain. Consider the subalgebra $\mathrm{Gr}(T) \subset \mathfrak{g} \ltimes_H M$ of the twisted semi-direct product. It follows from Proposition \ref{coboun} that $\Psi_h ( \mathrm{Gr}(T)) = \{ (Tu, u - hT(u)) | u \in M \} \subset \mathfrak{g} \ltimes_{H + \delta h} M$ is a subalgebra. If the linear map $(\mathrm{id} - h \circ T) : M \rightarrow M$ is invertible, then $\Psi_h ( \mathrm{Gr}(T))$ is the graph of the linear map $T (\mathrm{id} -  h \circ T)^{-1}$. In this case, it follows from Proposition \ref{graph-twisted} that $T (\mathrm{id} -  h \circ T)^{-1}$ is a $(H + \delta h)$-twisted Rota-Baxter operator.

\medskip

In the next, we give a construction of a new $H$-twisted Rota-Baxter operator out of an old one and a suitable $1$-cocycle. Let $T:M \rightarrow \mathfrak{g}$ be a $H$-twisted Rota-Baxter operator. Consider the graph $\mathrm{Gr}(T) = \{ (Tu, u) | u \in M \} \subset \mathfrak{g} \ltimes_H M$ which is a subalgebra of the $H$-twisted semi-direct product.

Let $B : \mathfrak{g} \rightarrow M$ be a $1$-cocycle in the Chevalley-Eilenberg cohomology of $\mathfrak{g}$ with coefficients in $M$. Since $B$ is a $1$-cocycle, the deformed subspace
\begin{align*}
\tau_B ( \mathrm{Gr}(T)) = \{ (Tu, u + BT (u))|~u \in M \} \subset \mathfrak{g} \ltimes_H M
\end{align*}
is still a subalgebra. This subalgebra may not be the graph of a linear map from $M$ to $\mathfrak{g}$. However, if the linear map $(\mathrm{id}+ B \circ T ) : M \rightarrow M$ is invertible, then $\tau_B ( \mathrm{Gr}(T))$ is the graph of the linear map $T (\mathrm{id}+ B \circ T )^{-1} : M \rightarrow \mathfrak{g}$. In this case, we say that $B$ is $T$-admissible $1$-cocycle. Then it follows from Proposition \ref{graph-twisted} that $T (\mathrm{id}+ B \circ T )^{-1}$ is a $H$-twisted Rota-Baxter operator. This is called the gauge transformation of $T$ associated with $B$. We denote this $H$-twisted Rota-Baxter operator by $T_B$.

\begin{prop}
Let $T$ be a $H$-twisted Rota-Baxter operator and $B$ be a $T$-admissible $1$-cocycle. Then the Lie algebra structures on $M$ induced by $T$ and $T_B$ are isomorphic.
\end{prop}

\begin{proof}
Consider the linear isomorphism $\mathrm{id} + B \circ T : M \rightarrow M$. We have
\begin{align*}
&[(\mathrm{id} + B \circ T)(u), (\mathrm{id} + B \circ T)(v)]_{T_B} \\
&= T(u) \bullet (\mathrm{id} + B \circ T) (v) - T(v) \bullet ( \mathrm{id} + B \circ T) (u) + H (Tu, Tv) \\
&= T(u) \bullet v - T(v) \bullet u + T(u) \bullet BT(v) - T(v) \bullet BT(u) + H (Tu, Tv) \\
&= T(u) \bullet v - T(v) \bullet u + B ( [Tu, Tv]) + H (Tu, Tv) \\
&= [u, v]_T + BT ( [u, v]_T) = (\mathrm{id} + B \circ T) ( [u,v]_T).
\end{align*}
Hence it follows that $(\mathrm{id} + B \circ T) : (M, [~,~]_T) \rightarrow (M, [~,~]_{T_B})$ is a Lie algebra isomorphism.
\end{proof}

\subsection{Reynolds operators}

Given a Lie algebra $(\mathfrak{g}, [~,~])$, denote the Lie bracket on $\mathfrak{g}$ by the map $\mu : \wedge^2 \mathfrak{g} \rightarrow \mathfrak{g}$. Note that $\mu$ is a $2$-cocycle in the Chevalley-Eilenberg of $\mathfrak{g}$ with  coefficients in itself. A $(-\mu)$-twisted Rota-Baxter operator is called a Reynolds operator on the Lie algebra $\mathfrak{g}$.

\begin{defn}
Let $(\mathfrak{g}, [~,~])$ be a Lie algebra. A Reynolds operator on $\mathfrak{g}$ is a linear map $R : \mathfrak{g} \rightarrow \mathfrak{g}$ satisfying
\begin{align}\label{reynolds-iden}
[Rx, Ry] = R ([Rx, y] + [x, Ry] - [Rx,Ry]),~ \text{ for } x, y \in \mathfrak{g}.
\end{align}
\end{defn}

 It follows from Proposition \ref{graph-twisted-new} that a Reynolds operator $R: \mathfrak{g} \rightarrow \mathfrak{g}$ induces a new Lie algebra structure $\mathfrak{g}$ given by
\begin{align}\label{rey-new-brkt}
[x, y]_R :=  [Rx, y] + [ x, Ry] - [Rx, Ry],~ \text{ for } x, y \in \mathfrak{g}.
\end{align}

\begin{exam}
Let $W$ be the Witt algebra generated by basis elements $\{ l_n \}_{n \in \mathbb{Z}}$ and the Lie bracket given by
\begin{align*}
[l_m , l_n ] = (m-n)~ l_{m+n},~ \text{ for } m, n \geq 0.
\end{align*}
Note that $W_{\geq 0} = \mathrm{span}\{ l_n | n \geq 0 \}$ is a Lie subalgebra of $W$. 
Then the linear map $R: W_{\geq 0} \rightarrow W_{\geq 0}$ defined by $R(l_m) = \frac{1}{m+1}~ l_m$, for $m \geq 0$, is a Reynolds operator on $W_{\geq 0}$.
First observe that, for any $m, n \geq 0$,
\begin{align}\label{witt1}
[R(l_m), R(l_n)] = \frac{1}{(m+1)(n+1)} ~[l_m, l_n] = \frac{(m-n)}{(m+1)(n+1)} ~l_{m+n}.
\end{align}
On the other hand, 
\begin{align}\label{witt2}
&R  \big( [R(l_m), l_n ] + [ l_m , R (l_n) ] - [R(l_m), R(l_n)] \big) \nonumber \\
&= R \bigg(   \frac{m-n}{m+1} ~l_{m+n} + \frac{m-n}{n+1}~ l_{m+n} - \frac{m-n}{(m+1)(n+1)}~ l_{m+n} \bigg) \nonumber \\
&= \bigg(   \frac{m-n}{m+1}  + \frac{m-n}{n+1}  - \frac{m-n}{(m+1)(n+1)} \bigg) \frac{1}{m+n+1} ~ l_{m+n} = \frac{(m-n)}{(m+1)(n+1)} ~ l_{m+n}.
\end{align}
It follows from (\ref{witt1}) and (\ref{witt2}) that $R$ is a Reynolds operator on $W_{\geq 0}$. The induced Lie algebra structure on $W_{\geq 0}$ is given by
\begin{align*}
[l_m, l_n]_R := [R(l_m), l_n ] + [ l_m , R (l_n) ] - [R(l_m), R(l_n)] = \frac{(m-n)(m+n+1)}{(m+1)(n+1)} ~l_{m+n}.
\end{align*}
\end{exam}

Let $R : \mathfrak{g} \rightarrow \mathfrak{g}$ be a Reynolds operator on $\mathfrak{g}$ such that $R$ is invertible. Then it follows from (\ref{reynolds-iden}) that
\begin{align*}
R^{-1} [u, v ] = [ u, R^{-1} v] + [R^{-1} u, v] - [u, v], ~ \text{ for } u, v \in \mathfrak{g}
\end{align*}
which implies that $(R^{-1} - \mathrm{id})[u,v] = [(R^{-1} - \mathrm{id})(u), v] + [ u, (R^{-1} - \mathrm{id})(v)]$, for $u, v \in \mathfrak{g}$. This shows that $(R^{-1} - \mathrm{id}): \mathfrak{g} \rightarrow \mathfrak{g}$ is a Lie algebra derivation. Conversely, if $d$ is a derivation such that $(\mathrm{id} + d) : \mathfrak{g} \rightarrow \mathfrak{g}$ is invertible, then $(\mathrm{id} + d)^{-1}$ is a Reynolds operator. Even, if $(\mathrm{id} + d)$ is not invertible but the infinite sum $(\mathrm{id} + d)^{-1} = \sum_{n =0}^\infty (-1)^n d^n$ converges pointwise, then $(\mathrm{id} + d)^{-1}$ is a Reynolds operator. More precise statement is given below whose proof is similar to \cite[Proposition 2.8]{zhang-gao-guo}.

\begin{prop}
Let $\mathfrak{g}$ be a Lie algebra and $d : \mathfrak{g} \rightarrow \mathfrak{g}$ be a derivation. Suppose, for each $x \in \mathfrak{g}$, the infinite sum $\sum_{n =0}^\infty (-1)^n d^n (x)$  converges  to a element in $\mathfrak{g}$. Then $R = \sum_{n =0}^\infty (-1)^n d^n$ is a Reynolds operator on $\mathfrak{g}$.
\end{prop}

It follows from the above proposition that if $d$ is a nilpotent derivation (more generally a locally nilpotent derivation) on $\mathfrak{g}$, then $R = \sum_{n=0}^\infty (-1)^n d^n$ is a Reynolds operator on $\mathfrak{g}$.

\section{Maurer-Cartan characterization of twisted Rota-Baxter \\ operators and cohomology}\label{sec-mc}
In this section, we construct an $L_\infty$-algebra whose Maurer-Cartan elements are $H$-twisted Rota-Baxter operators. This $L_\infty$-algebra consists of only bilinear and trilinear brackets, among which the bilinear bracket (which is constructed in \cite{tang} to characterize Rota-Baxter operators) itself satisfies the graded Jacobi identity. Such characterization of $H$-twisted Rota-Baxter operator $T$ allows us to introduce the cohomology of $T$. We also show that the cohomology of $T$ is equivalently described by the Chevalley-Eilenberg cohomology of $M$ with coefficients in a suitable representation of $\mathfrak{g}$.

\begin{defn} \cite{lada-markl,lada-stasheff}
An $L_\infty$-algebra is a graded vector space $L = \bigoplus_i L_i$ together with a collection $\{ l_k : L^{\otimes k} \rightarrow L |~ \mathrm{deg}(l_k) = 2-k \}_{k \geq 1}$ of multilinear maps satisfying the followings.
\begin{itemize}
\item[(i)] Skew-symmetry:~ $l_k ( x_{\sigma(1)} , \ldots, x_{\sigma (k)} ) = (-1)^\sigma \epsilon(\sigma) l_k ( x_1, \ldots, x_k),$  for $k \geq 1$ and $\sigma \in S_k$.
Here $\epsilon(\sigma) = \epsilon (\sigma; x_1, \ldots, x_k)$ is the Koszul sign.
\item[(ii)] Higher Jacobi identities: for each $n \geq 1$, we have
\begin{align*}
\sum_{i+j = n+1}^{} \sum_{\sigma}^{} (-1)^\sigma \epsilon (\sigma) ~ (-1)^{i (j-1)} ~ l_j   \big(  l_i (x_{\sigma (1)}, \ldots, x_{\sigma (i)}),  x_{\sigma (i+1)}, \ldots, x_{\sigma (n)} \big) = 0,
\end{align*}
where $\sigma$ runs over all $(i, n-i)$-unshuffles with $i \geq 1$.
\end{itemize}
\end{defn}

An element $\alpha \in L_1$ is said to be a Maurer-Cartan element (see for instance \cite{markl}) if it satisfies
\begin{align*}
l_1 (\alpha) + \frac{1}{2!} l_2 (\alpha, \alpha) - \frac{1}{3!} l_3 (\alpha, \alpha, \alpha ) - \cdots = 0.
\end{align*}



Next we recall the graded Lie algebra constructed in \cite{tang}. Let $\mathfrak{g}$ be a Lie algebra and $M$ be a $\mathfrak{g}$-module. Denote the Lie bracket on $\mathfrak{g}$ by the map $\mu : \wedge^2 \mathfrak{g} \rightarrow \mathfrak{g}$ and the $\mathfrak{g}$-module structure on $M$ by $\bullet : \mathfrak{g} \times M \rightarrow M$. Then the graded vector space  $\bigoplus_{n \geq 0} \mathrm{Hom}(\wedge^n M, \mathfrak{g})$ carries a graded Lie bracket given by
\begin{align*}
\llbracket P, Q \rrbracket := (-1)^p~ [[\mu + \bullet , P]_{\mathsf{NR}}, Q]_{\mathsf{NR}}, 
\end{align*}
for $P \in \mathrm{Hom}(\wedge^p M, \mathfrak{g})$, $Q \in \mathrm{Hom}( \wedge^q M, \mathfrak{g})$, and $[~, ~]_{\mathsf{NR}}$ is the Nijenhuis-Richardson bracket on skew-symmetric multilinear maps on the vector space $\mathfrak{g} \oplus M$. Explicitly, 
\begin{align}\label{l_2}
\llbracket P, Q \rrbracket (u_1, \ldots, u_{p+q})
=~& \sum_{\sigma \in Sh (q, 1, p-1)} (-1)^\sigma ~ P (  Q (u_{\sigma (1)} , \ldots, u_{\sigma (q)}) \bullet u_{\sigma (q+1)}, u_{\sigma (q+2)}, \ldots, u_{\sigma (p+q)} ) \\
~&- (-1)^{pq} \sum_{\sigma \in Sh (p, 1, q-1)} (-1)^\sigma ~ Q (  P (u_{\sigma (1)} , \ldots, u_{\sigma (p)}) \bullet u_{\sigma (p+1)}, u_{\sigma (p+2)}, \ldots, u_{\sigma (p+q)} ) \nonumber \\
~&+ (-1)^{pq} \sum_{Sh (p, q)} (-1)^\sigma [ P (u_{\sigma (1)} , \ldots, u_{\sigma (p)}), Q (u_{\sigma (n+1)}, \ldots, u_{\sigma (p+q)}) ]. \nonumber 
\end{align}
Hence, for any $T \in \mathrm{Hom}(M, \mathfrak{g})$,
\begin{align}\label{binary-t}
\llbracket T, T \rrbracket (u,v) = 2~ \{ T ( T(u) \bullet v - T(v) \bullet u) - [Tu, Tv] \}, \text{ for } u, v \in M.
\end{align}
This shows that Maurer-Cartan elements on this graded Lie algebra are Rota-Baxter operators \cite{tang}.

\medskip

Let $H$ be a $2$-cocycle in the cohomology of $\mathfrak{g}$ with coefficients in $M$. Following the case of twisted Poisson manifolds \cite{fre-zam}, here we introduce a ternary degree $-1$ bracket $\llbracket ~, ~, ~ \rrbracket$ on the graded space $\bigoplus_{n \geq 0} \mathrm{Hom}(\wedge^n M , \mathfrak{g})$ as follows. For  $P \in \mathrm{Hom}(\wedge^p M, \mathfrak{g})$, $Q \in \mathrm{Hom}(\wedge^q M, \mathfrak{g})$, and $R \in \mathrm{Hom}(\wedge^r M, \mathfrak{g})$, we define $\llbracket P, Q, R \rrbracket \in \mathrm{Hom}(\wedge^{p+q+r-1} M, \mathfrak{g})$ by
\begin{align}\label{l_3}
&\llbracket P, Q, R \rrbracket (u_1, \ldots, u_{p+q+r-1}) \\
&= (-1)^{pqr} ~\frac{1}{2}  \bigg\{ \sum_{\sigma \in Sh (q, r, p-1)} (-1)^\sigma ~P \big( H (   Q ( u_{\sigma (1)}, \ldots, u_{\sigma (q)}), R (u_{\sigma (q+1)}, \ldots, u_{\sigma (q+r)}) ), u_{\sigma (q+r+1)} \ldots, u_{\sigma (p+q+r-1)}    \big) \nonumber \\
&- (-1)^{qr} \sum_{\sigma \in Sh (r, q, p-1)} (-1)^\sigma ~P \big( H (   R ( u_{\sigma (1)}, \ldots, u_{\sigma (r)}), Q (u_{\sigma (r+1)}, \ldots, u_{\sigma (q+r)}) ), u_{\sigma (q+r+1)} \ldots, u_{\sigma (p+q+r-1)}    \big) \nonumber \\
&- (-1)^{pq} \sum_{\sigma \in Sh (p, r, q-1)} (-1)^\sigma ~Q \big( H (   P ( u_{\sigma (1)}, \ldots, u_{\sigma (p)}), R (u_{\sigma (p+1)}, \ldots, u_{\sigma (p+r)}) ), u_{\sigma (p+r+1)} \ldots, u_{\sigma (p+q+r-1)}    \big) \nonumber \\
& + (-1)^{p(q+r)} \sum_{\sigma \in Sh (r, p, q-1)} (-1)^\sigma ~Q \big( H (   R ( u_{\sigma (1)}, \ldots, u_{\sigma (r)}), P (u_{\sigma (r+1)}, \ldots, u_{\sigma (p+r)}) ), u_{\sigma (p+r+1)} \ldots, u_{\sigma (p+q+r-1)}    \big) \nonumber \\
&+ (-1)^{(p+q)r} \sum_{\sigma \in Sh (p, q, r-1)} (-1)^\sigma ~R \big( H (   P ( u_{\sigma (1)}, \ldots, u_{\sigma (p)}), Q (u_{\sigma (p+1)}, \ldots, u_{\sigma (p+q)}) ), u_{\sigma (p+q+1)} \ldots, u_{\sigma (p+q+r-1)}    \big) \nonumber \\
&- (-1)^{pq+qr+rp} \sum_{\sigma \in Sh (q, p, r-1)} (-1)^\sigma ~R \big( H (   Q ( u_{\sigma (1)}, \ldots, u_{\sigma (q)}), P(u_{\sigma (q+1)}, \ldots, u_{\sigma (p+q)}) ), u_{\sigma (p+q+1)} \ldots, u_{\sigma (p+q+r-1)}    \big) \bigg\}. \nonumber
\end{align}
The bracket $\llbracket P, Q, R \rrbracket$ is obviously graded skew-symmetric. Moreover, the binary bracket $\llbracket ~, ~ \rrbracket$ and the ternary bracket $\llbracket ~, ~, ~ \rrbracket$ are compatible in the sense of $L_\infty$-algebra. This follows as $H$ is a $2$-cocycle. In summary, we obtain the following.

\begin{thm}\label{mc-thm}
Let $\mathfrak{g}$ be a Lie algebra, $M$ be a $\mathfrak{g}$-module and $H$ be a $2$-cocycle in the cohomology of $\mathfrak{g}$ with coefficients in $M$. Then the graded vector space $\bigoplus_{n \geq 0} \mathrm{Hom}(\wedge^n M, \mathfrak{g})$ is an $L_\infty$-algebra with
\begin{align*}
l_1 = 0, \qquad \quad l_2 = \llbracket ~, ~ \rrbracket, \qquad \quad l_3 = \llbracket ~, ~, ~ \rrbracket,
\end{align*}
and higher maps are trivial. A linear map $T : M \rightarrow \mathfrak{g}$ is a $H$-twisted Rota-Baxter operator if and only if $T \in \mathrm{Hom}(M, \mathfrak{g})$ is a Maurer-Cartan element in the above $L_\infty$-algebra $( \bigoplus_{n \geq 0} \mathrm{Hom}(\wedge^n M, \mathfrak{g}), \llbracket ~, ~ \rrbracket ,  \llbracket ~, ~, ~ \rrbracket )$.
\end{thm}

\begin{proof} The first part follows from the previous discussions. For the second part, we observe that for any $T \in \mathrm{Hom} (M, \mathfrak{g})$,
\begin{align}\label{ternary-t}
\llbracket T, T, T \rrbracket (u, v) = - 6 ~ T (H (Tu, Tv)).
\end{align}
Hence from (\ref{binary-t}) and (\ref{ternary-t}), we get
\begin{align*}
\big( \frac{1}{2} \llbracket T, T \rrbracket - \frac{1}{6} \llbracket T, T, T \rrbracket \big) (u, v) =   T (T(u) \bullet v - T(v) \bullet u ) - [Tu, Tv] + T (H (Tu, Tv)).
\end{align*}
This shows that $T$ is a $H$-twisted Rota-Baxter operator if and only if $T$ is a Maurer-Cartan element in the $L_\infty$-algebra.
\end{proof}

\medskip

The above characterization of a $H$-twisted Rota-Baxter operator $T$ allows us to define a cohomology associated to $T$. More precisely, we define $C^n_T (M, \mathfrak{g}) = \mathrm{Hom}(\wedge^n M, \mathfrak{g})$, for $n \geq 0$, and the differential $d_T : C^n_T (M, \mathfrak{g}) \rightarrow C^{n+1}_T (M, \mathfrak{g})$ by
\begin{align*}
d_T (f ) = \llbracket T, f \rrbracket - \frac{1}{2} \llbracket T, T, f \rrbracket, ~ \text{ for } f \in C^n_T (M, \mathfrak{g}).
\end{align*}
The corresponding cohomology groups
\begin{align*}
H^n_T (M, \mathfrak{g}) = \frac{Z^n_T (M, \mathfrak{g})}{B^n_T (M, \mathfrak{g})} = \frac{ \{f \in C^n_T (M, \mathfrak{g}) |~ d_T f = 0\} }{ \{d_T g |~ g \in C^{n-1}_T (M, \mathfrak{g}) \} } , \text{ for } n \geq 0
\end{align*}
are called the cohomology of the $H$-twisted Rota-Baxter operator $T$.

Given an $L_\infty$-algebra and a Maurer-Cartan element, Getzler \cite{getzler} construct a new $L_\infty$-algebra twisted by the Maurer-Cartan element. Here we will follow the conventions of \cite{markl}. This rephrases in our context as follows.

\begin{prop}
Let $T$ be a $H$-twisted Rota-Baxter operator. Then $\bigoplus_{n \geq 0 } \mathrm{Hom}(\wedge^n M, \mathfrak{g})$ carries an $L_\infty$-algebra with structure maps
\begin{align*}
l_1 (P ) = d_T (P), \qquad \llbracket P, Q \rrbracket_T = \llbracket P, Q \rrbracket - \llbracket T, P, Q \rrbracket, \qquad \llbracket P, Q, R \rrbracket_T = \llbracket P, Q, R \rrbracket
\end{align*}
and trivial higher brackets. We call this as the twisted $L_\infty$-algebra twisted by $T$.
\end{prop}

\begin{thm}
Let $T$ be a $H$-twisted Rota-Baxter operator. For any linear map $T' : M \rightarrow \mathfrak{g}$, the sum $T + T'$ is a $H$-twisted Rota-Baxter operator if and only if $T'$ is a Maurer-Cartan element in the above twisted $L_\infty$-algebra .
\end{thm}

\begin{proof}
We have
\begin{align*}
&\frac{1}{2} \llbracket T + T' , T + T' \rrbracket - \frac{1}{6} \llbracket T + T', T + T', T + T' \rrbracket \\
&= \frac{1}{2} \llbracket T, T' \rrbracket + \frac{1}{2} \llbracket T', T \rrbracket + \frac{1}{2} \llbracket T', T' \rrbracket  - \frac{1}{6} \llbracket T, T, T' \rrbracket -  \frac{1}{6} \llbracket T, T', T \rrbracket -  \frac{1}{6} \llbracket T', T, T \rrbracket \\
&- \frac{1}{6} \llbracket T, T', T' \rrbracket -  \frac{1}{6} \llbracket T', T, T' \rrbracket -  \frac{1}{6} \llbracket T', T', T \rrbracket - \frac{1}{6} \llbracket T', T', T' \rrbracket \\
&= (\llbracket T, T' \rrbracket - \frac{1}{2} \llbracket T, T, T' \rrbracket)   + \frac{1}{2} ( \llbracket T', T' \rrbracket - \llbracket T, T', T' \rrbracket )  - \frac{1}{6} \llbracket T', T', T' \rrbracket \\
&= d_T (T') + \frac{1}{2} \llbracket T', T' \rrbracket_T - \frac{1}{6} \llbracket T', T', T' \rrbracket_T.
\end{align*}
This proves the statement.
\end{proof}

\subsection{Another interpretation of the cohomology}

Let $T : M \rightarrow \mathfrak{g}$ be a $H$-twisted Rota-Baxter operator. We have seen in Proposition \ref{graph-twisted-new} that the vector space $M$ carries a Lie algebra structure with bracket
\begin{align*}
[u,v]_T := T(u) \bullet v - T(v) \bullet u + H (Tu, Tv),~ \text{ for } u, v \in M.
\end{align*}
In this subsection, we show that the cohomology of $T$ can be interpreted as the Chevalley-Eilenberg cohomology of $(M, [~,~]_T)$ with coefficients in a suitable module structure on $\mathfrak{g}$.

We start with the following.

\begin{prop}
The linear map $\overline{\bullet } : M \times \mathfrak{g} \rightarrow \mathfrak{g}$ defined by
\begin{align*}
u ~\overline{\bullet}~ x = [ Tu, x ] + T (x \bullet u + H (x, Tu)), ~ \text{ for } u \in M, x \in \mathfrak{g}
\end{align*}
is a representation of the Lie algebra $(M, [~, ~]_T)$ on the vector space $\mathfrak{g}$.
\end{prop}

\begin{proof}
By a direct calculation using the definition of $\overline{\bullet}$, we get for $u, v \in M,~ x \in \mathfrak{g},$
\begin{align*}
&u ~\overline{\bullet}~ ( v ~\overline{\bullet}~ x )  - v ~\overline{\bullet}~ ( u ~\overline{\bullet}~ x) - [u, v]_T ~\overline{\bullet}~ x \\
&= \cancel{[ Tu, [Tv, x]]} + T ( [Tv, x] \bullet u) + TH ( [Tv, x], Tu) + [Tu, T ( x \bullet v)] + T ( T (x \bullet v) \bullet u) \\ &~~+ TH ( T (x \bullet v) , Tu) 
+ [Tu, TH (x, Tv)] + T ( TH (x, Tv) \bullet u) + TH (TH (x, Tv), Tu) - \cancel{[Tv, [Tu, x]]}  \\ &~~- T ( [Tu, x] \bullet v) - TH ([Tu, x], Tv)
- [Tv, T (x \bullet u)] - T ( T (x \bullet u) \bullet v) - TH ( T (x \bullet u), Tv) \\ &~~ - [Tv, TH (x, Tu)] - T ( TH (x, Tu) \bullet v ) -TH (TH (x, Tu), Tv) 
 -  \cancel{[[Tu, Tv], x]} - T ( x \bullet (Tu \bullet v)) \\ &~~ + T ( x \bullet (Tv \bullet u)) - T ( x \bullet H (Tu, Tv)) - TH ( x, [Tu, Tv])\\
&= T ( [Tv, x] \bullet u) + TH ( [Tv, x],Tu) + T ( (Tu) \bullet (x \bullet v))  + TH (Tu, T ( x \bullet v))  + TH ( T ( x \bullet v), Tu) \\
&~~+ T ( Tu \bullet H (x, Tv)) + TH (Tu, TH (x, Tv)) + TH (TH (x, Tv), Tu) - T ( [Tu, x] \bullet v ) - TH ([Tu, x] , Tv)\\
&~~ - T ( (Tv) \bullet (x \bullet u)) - TH ( Tv, T (x \bullet u)) - TH ( T ( x \bullet u ), Tv) - T ( Tv \bullet H (x, Tu)) - TH ( Tv, TH (x, Tu)) \\&~~ - TH ( TH (x, Tu), Tv) - T ( x \bullet (Tu \bullet v )) + T ( x \bullet (Tv \bullet u)) - T ( x \bullet H (Tu, Tv)) - TH (x, [Tu, Tv]) \\
&\stackrel{\text{after cancellations}}{=} TH ([Tv, x], Tu) + T (Tu \bullet H (x, Tv)) - TH ( [Tu, x], Tv) - T ( Tv \bullet H (x, Tu)) \\
&~~ - T ( x \bullet H (Tu, Tv)) - TH (x, [Tu, Tv])) \\
&= - T \big( (\delta_{\mathrm{CE}} H) (Tu, Tv, x) \big) = 0. 
\end{align*}
Hence the result follows.
\end{proof}

It follows from the above proposition that we may consider the Chevalley-Eilenberg cohomology of the Lie algebra $(M, [~,~]_T)$ with coefficients in the representation $\mathfrak{g}$. More precisely, we define
\begin{align*}
C^n_{\mathrm{CE}}(M, \mathfrak{g}) := \mathrm{Hom}(\wedge^n M, \mathfrak{g}), ~ \text{ for } n \geq 0
\end{align*}
and the differential $\delta_{\mathrm{CE}} : C^n_{\mathrm{CE}} (M, \mathfrak{g}) \rightarrow C^{n+1}_{\mathrm{CE}}(M, \mathfrak{g})$ by
\begin{align*}
&(\delta_{\mathrm{CE}} f) (u_1, \ldots, u_{n+1}) \\
&= \sum_{i=1}^{n+1} (-1)^{i+1} ~[ Tu_i , f(u_1, \ldots, \widehat{u_i}, \ldots, u_{n+1}) ] ~+~ \sum_{i=1}^{n+1} (-1)^{i+1} ~ T ( f(u_1, \ldots, \widehat{u_i}, \ldots, u_{n+1}) \bullet u_i ) \\
&+ \sum_{i=1}^{n+1} (-1)^{i+1} ~ TH ( f(u_1, \ldots, \widehat{u_i}, \ldots, u_{n+1}), Tu_i ) \\
&+ \sum_{i \leq i < j \leq n+1} (-1)^{i+j} ~ f (  T(u_i) \bullet u_j - T(u_j) \bullet u_i + H (Tu_i, Tu_j ), u_1, \ldots, \widehat{u_i}, \ldots, \widehat{ u_j}, \ldots, u_{n+1} ),
\end{align*}
for $f \in C^n_{\mathrm{CE}}(M, \mathfrak{g})$ and $u_1, \ldots, u_{n+1} \in M$.

\begin{prop}
Let $T: M \rightarrow \mathfrak{g}$ be a $H$-twisted Rota-Baxter operator. Then for any $f \in \mathrm{Hom}(\wedge^n M, \mathfrak{g})$,
\begin{align*}
d_T f = (-1)^n \ \delta_{\mathrm{CE}} f.
\end{align*}
Consequently, the cohomology of $T$ is isomorphic to the Chevalley-Eilenberg cohomology of $(M, [~,~]_*)$ with coefficients in $\mathfrak{g}$.
\end{prop}

\begin{proof}
In \cite[Proposition 3.3]{tang} the authors showed that
\begin{align*}
&\llbracket T, f \rrbracket (u_1, \ldots, u_{n+1}) \\
&= (-1)^n \bigg\{  \sum_{i=1}^{n+1} (-1)^{i+1} ~[ Tu_i , f(u_1, \ldots, \widehat{u_i}, \ldots, u_{n+1}) ] ~+~ \sum_{i=1}^{n+1} (-1)^{i+1} ~ T ( f(u_1, \ldots, \widehat{u_i}, \ldots, u_{n+1}) \bullet u_i ) \\
& \qquad \qquad \quad + \sum_{1 \leq i < j \leq n+1} (-1)^{i+j} ~ f (  T(u_i) \bullet u_j - T(u_j) \bullet u_i, u_1, \ldots, \widehat{u_i}, \ldots, \widehat{ u_j}, \ldots, u_{n+1} ) \bigg\}.
\end{align*}
Here we observe that 
\begin{align*}
\llbracket T, T, f \rrbracket (u_1, \ldots, u_{n+1})
&= (-1)^n \frac{1}{2} \bigg\{ 2 \sum_{i=1}^{n+1} (-1)^{i+1}~ T ( H (Tu_i, f (u_1, \ldots, \widehat{u_i}, \ldots, u_{n+1}))) \\
& \qquad \qquad \quad - (-1)^n~ 2 \sum_{i=1}^{n+1} (-1)^{n-i+1} ~ T  ( H (  f (u_1, \ldots, \widehat{u_i}, \ldots, u_{n+1}) , Tu_i )  )    \\
& \qquad \qquad \quad + 2 \sum_{1 \leq i < j \leq n+1} (-1)^{i+j-1} ~ f ( H (Tu_i, Tu_j ), u_1, \ldots, \widehat{u_i}, \ldots, \widehat{ u_j} , \ldots, u_{n+1}) \\
& \qquad \qquad \quad + 2 \sum_{1 \leq i < j \leq n+1} (-1)^{i+j} ~ f ( H (Tu_j, Tu_i ), u_1, \ldots, \widehat{u_i}, \ldots, \widehat{ u_j} , \ldots, u_{n+1})  \bigg\}\\
&= - (-1)^n ~2 \bigg\{ \sum_{i=1}^{n+1} (-1)^{i+1} ~ TH ( f(u_1, \ldots, \widehat{u_i}, \ldots, u_{n+1}), Tu_i ) \\
& \qquad \qquad \quad + \sum_{i \leq i < j \leq n+1} (-1)^{i+j} ~ f ( H (Tu_i, Tu_j ), u_1, \ldots, \widehat{u_i}, \ldots, \widehat{ u_j}, \ldots, u_{n+1} ) \bigg\}.
\end{align*}
Hence 
$d_T f = \llbracket T, f \rrbracket - \frac{1}{2} \llbracket T, T, f \rrbracket = (-1)^n ~ \delta_{\mathrm{CE}} f.$
This proves the result.
\end{proof}

\section{Deformations of twisted Rota-Baxter operators}\label{sec-4}
In this section, we study linear and formal deformations of a $H$-twisted Rota-Baxter operator. We introduce Nijenhuis elements associated with a $H$-twisted Rota-Baxter operator that arise from trivial deformations. Finally, we consider rigidity of a $H$-twisted Rota-Baxter operator and provide a sufficient condition for the same in terms of Nijenhuis elements.
  

\subsection{Linear deformations}
Let $\mathfrak{g}$ be a Lie algebra, $M$ be a $\mathfrak{g}$-module and $H : \wedge^2 \mathfrak{g} \rightarrow M$ be a $2$-cocycle in the Chevalley-Eilenberg cohomology of $\mathfrak{g}$ with coefficients in $M$. Suppose $T: M \rightarrow \mathfrak{g}$ is a $H$-twisted Rota-Baxter operator.

\begin{defn}
A linear deformation of $T$ consists of a parametrized sum $T_t = T + t T_1$, for some $T_1 \in \mathrm{Hom}(M, \mathfrak{g})$ such that $T_t$ is a $H$-twisted Rota-Baxter operator for all values of $t$. In this case, we say that $T_1$ generates a linear deformation of $T$.
\end{defn}

Thus, in a linear deformation $T_t = T + t T_1$, we must have the followings
\begin{align}
[T(u), T_1(v)] + [T_1(u), T(v)] =~& T_1 ( T(u) \bullet v - T(v) \bullet u + H (Tu, Tv)) \label{lin-def-a} \\
&+ T \big( T_1(u) \bullet v - T_1(v) \bullet u + H (T_1(u), Tv) + H (Tu, T_1(v)) \big), \nonumber \\
[T_1(u), T_1(v)] =~& T ( H ( T_1(u), T_1(v))) + T_1 ( T_1(u) \bullet v - T_1(v) \bullet u + H (T_1(u), Tv) + H (Tu, T_1(v)) ), \\
0 =~& T_1 (H(T_1(u), T_1(v)) ).
\end{align}

Note that the identity (\ref{lin-def-a}) implies that $T_1$ is a $1$-cocycle in the cohomology of $T$. Hence, $T_1$ defines a cohomology class in $H^1_T (M, \mathfrak{g}).$

\begin{defn}
Two linear deformations $T_t = T+ tT_1$ and $T_t' = T + t T_1'$ of a $H$-twisted Rota-Baxter operator $T$ are said to be equivalent if there exists an element $x \in \mathfrak{g}$ such that the pair
\begin{align*}
\big( \phi_t = \mathrm{id}_\mathfrak{g} + t [x, -],~ \psi_t = \mathrm{id}_M + t (x \bullet - ~+ H (x, T -)) \big)
\end{align*}
defines a morphism of $H$-twisted Rota-Baxter operators from $T_t$ to $T_t'$.
\end{defn}

Thus, it follows that $\phi_t$ is a Lie algebra homomorphism which implies that
\begin{align}\label{def-lie-hom}
[[x,y], [x,z]] = 0, ~ \text{ for } y, z \in \mathfrak{g}.
\end{align}
The condition $\psi_t ( y \bullet u ) = \phi_t (y) \bullet \psi_t (u)$ implies that
\begin{align}\label{def-action-pre}
\begin{cases}
H (x, T( y \bullet u)) = y \bullet H (x, Tu),\\
[x,y] \bullet (x \bullet u + H (x, Tu)) = 0.
\end{cases}
\end{align}
Finally, the conditions $\psi_t \circ H = H \circ (\phi_t \otimes \phi_t )$ and $\phi_t \circ T_t = T_t' \circ \psi_t$ are respectively equivalent to
\begin{align}\label{def-new-1}
\begin{cases}
x \bullet H (y,z) + H (x, TH (y,z)) = H (  [x,y], z ) + H (y, [x,z]),\\
H ( [x,y] , [x, z]) = 0,
\end{cases}
\end{align}
\begin{align}\label{def-new-2}
\begin{cases}
T_1 (u) + [x, Tu] = T (x \bullet u  + H (x, Tu)) + T_1' (u),\\
[x, T_1(u)]  = T_1' (x \bullet u + H (x, Tu)).
\end{cases}
\end{align}
Note that the above identities hold for all $y, x \in \mathfrak{g}$ and $u \in M$. It is easy to see from the 
 first condition of (\ref{def-new-2}) that $T_1(u) - T_1'(u) = d_T(x)(u)$. Therefore, we get the following.
\begin{thm}
Let $T_t = T + t T_1$ and $T_t = T+ t T_1'$ be two equivalent linear deformations of a $H$-twisted Rota-Baxter operator. Then $T_1$ and $T_1'$ defines the same cohomology class in $H^1_T(M, \mathfrak{g})$.
\end{thm}

\begin{defn}
A linear deformation $T_t = T + t T_1$ of a $H$-twisted Rota-Baxter operator is said to be trivial if $T_t$ is equivalent to the undeformed deformation $T_t' = T$.
\end{defn} 

We will now define Nijenhuis elements associated with a $H$-twisted Rota-Baxter operator.

\begin{defn}
Let $T$ be a $H$-twisted Rota-Baxter operator. An element $x \in \mathfrak{g}$ is called a Nijenhuis element associated to $T$ if $x$ satisfies
\begin{align*}
[x, u \overline{\bullet} x ] = 0
\end{align*}
and (\ref{def-lie-hom}), (\ref{def-action-pre}), (\ref{def-new-1}) holds.
\end{defn}

We denote by $\mathrm{Nij}(T)$ the set of all Nijenhuis elements associated with the $H$-twisted Rota-Baxter operator. It follows from the above definition that a trivial linear deformation gives rise to a Nijenhuis element which is the motivation to introduce the definition. In the next subsection, we obtain a sufficient condition for the rigidity of a twisted Rota-Baxter operator in terms of Nijenhuis elements.




\subsection{Formal deformations}
In this subsection, we consider formal deformations of $H$-twisted Rota-Baxter operators generalizing the classical deformation theory of Gerstenhaber \cite{gers}.

Let $\mathfrak{g}$ be a Lie algebra, $M$ be a $\mathfrak{g}$-module and $H$ be a $2$-cocycle in the Chevalley-Eilenberg cohomology of $\mathfrak{g}$ with coefficients in $M$. Let $T: M \rightarrow \mathfrak{g}$ be a $H$-twisted Rota-Baxter operator.

Consider the space $\mathfrak{g}[[t]]$ of formal power series in $t$ with coefficients from $\mathfrak{g}$. Then $\mathfrak{g}[[t]]$ is a $\mathbb{K}[[t]]$-module. Note that the Lie algebra structure on $\mathfrak{g}$ induces a Lie algebra structure on $\mathfrak{g}[[t]]$ by $\mathbb{K}[[t]]$-bilinearity. Moreover, the $\mathfrak{g}$-module structure on $M$ induces a $\mathfrak{g}[[t]]$-module structure on $M[[t]]$. Similarly, the $2$-cocycle $H$ can be extended to a $2$-cocycle (denoted by the same notation $H$) on the Lie algebra $\mathfrak{g}[[t]]$ with coefficients in $M[[t]]$.

\begin{defn}
A formal deformation of $T$ consists of a formal sum $T_t = \sum_{i=0}^\infty t^i T_i$ (with $T_i \in \mathrm{Hom}(M, \mathfrak{g})$ for each $i$ and $T_0 = T$) such that $T_t$ is a $H$-twisted Rota-Baxter operator.
\end{defn}

Thus, in a formal deformation as above, the following system of equations must hold: for $n \geq 0$,
\begin{align}\label{formal-def-eqn}
\sum_{i+j = n} [T_i (u), T_j (v)] = \sum_{i+j = n} T_i ( T_j (u) \bullet v - T_j (v) \bullet u) + \sum_{i+j+k = n} T_i \big( H(T_j(u), T_k(v)) \big).
\end{align}
These equations are called deformation equations. Note that (\ref{formal-def-eqn}) holds for $n = 0$ as $T_0 = T$ is a $H$-twisted Rota-Baxter operator. For $n = 1$, we get
\begin{align*}
[T(u), T_1(v)] + [T_1(u), T(v)] =~& T_1 ( T(u) \bullet v - T(v) \bullet u + H (Tu, Tv)) \\
&+ T \big( T_1(u) \bullet v - T_1(v) \bullet u + H (T_1(u), Tv) + H (Tu, T_1(v)) \big).
\end{align*}
This implies that $(d_T (T_1))(u, v) = 0$. Hence the linear term $T_1$ is a $1$-cocycle in the cohomology of $T$, called the infinitesimal of the deformation $T_t$.

In the following, we define an equivalence between two formal deformations generalizing the equivalence between linear deformations.

\begin{defn}
Let $T_t = \sum_{i=0}^\infty t^i T_i$ and $T_t' = \sum_{i=0}^\infty t^i T_i'$ be two formal deformations of a $H$-twisted Rota-Baxter operator. They are said to be equivalent if there exists an element $x \in \mathfrak{g}$, linear maps $\phi_i \in \mathrm{Hom}(\mathfrak{g}, \mathfrak{g})$ and $\psi \in \mathrm{Hom}(M, M)$ for $i \geq 2$ such that the pair
\begin{align*}
(\phi_t = \mathrm{id}_\mathfrak{g} + t [x, -] + \sum_{i =2}^\infty t^i \phi_i,~ \psi_t = \mathrm{id}_M + t (x \bullet - ~ + H(x, T -)) + \sum_{i=2}^\infty t^i \psi_i )
\end{align*}
is a morphism of $H$-twisted Rota-Baxter operators from $T_t$ to $T_t'$.
\end{defn}

Thus, it follows that $(\phi_t, \psi_t)$ satisfies the conditions of Definition \ref{rota-morphism}. In particular, the condition $(\phi_t \circ T_t)(u) = (T_t' \circ \psi_t) (u)$ implies that (by equating coefficients of $t$)
\begin{align*}
T_1 (u) - T_1'(u) = [Tu, x] + T ( x \bullet u + H (x, Tu)) = (d_T (x)) (u).
\end{align*}
Hence, we get the following.

\begin{thm}
If $T_t = \sum_{i=0}^\infty t^i T_i$ and $T_t' = \sum_{i=0}^\infty t^i T_i'$ are two equivalent formal deformations of a $H$-twisted Rota-Baxter operator $T$, then $T_1$ and $T_1'$ defines the same cohomology class in $H^1_T (M, \mathfrak{g}).$
\end{thm}

\begin{defn}
A $H$-twisted Rota-Baxter operator $T$ is said to be rigid if any formal deformation of $T$ is equivalent to the undeformed one $T_t' = T$.
\end{defn}

The next result gives a sufficient condition for the rigidity of a $H$-twisted Rota-Baxter operator.

\begin{thm}
Let $T$ be a $H$-twisted Rota-Baxter operator. If $Z^1_T (M, \mathfrak{g}) = d_T (\mathrm{Nij}(T))$ then $T$ is rigid.
\end{thm}

\begin{proof}
Let $T_t = \sum_{i=0}^\infty t^i T_i$ be any formal deformation of $T$. Then we have seen that $T_1$ is a $1$-cocycle in the cohomology of $T$, i.e. $T_1 \in Z^1_T (M, \mathfrak{g})$. Thus, from the hypothesis, we get a Nijenhuis element $x \in \mathrm{Nij}(T)$ such that $T_1 = d_T (x)$. We take
\begin{align*}
\phi_t = \mathrm{id}_\mathfrak{g} + t [x, -], \qquad \psi_t = \mathrm{id}_M + t ( x \bullet - ~ + H (x, T-)),
\end{align*}
and define $T_t' = \phi_t \circ T_t \circ \psi_t^{-1}$. Then $T_t'$ is a deformation equivalent to $T_t$. We claim that the coefficient of $t$ in $T_t'$ is trivial. To see this, we observe that
\begin{align*}
T_t'(u) ~~~(\mathrm{mod }~ t^2) =~&    (\mathrm{id}_\mathfrak{g} + t [x, -]) \circ (T + tT_1) \circ (\mathrm{id}_M - t (x \bullet - ~ + H (x, T-))(u) ~~~(\mathrm{mod }~ t^2) \\
=~&   (\mathrm{id}_\mathfrak{g} + t [x, -]) ( Tu + t T_1(u) - t T(x \bullet u + H (x, Tu)))  ~~~(\mathrm{mod }~ t^2) \\
=~& Tu + t (     T_1 (u) - T(x \bullet u)  -TH (x, Tu) + [x, Tu]) \\
=~& Tu ~~~~ (\mathrm{as }~ T_1 (u) = d_T (x)(u)).
\end{align*}
This proves the claim. By applying the same process repeatedly, we get that $T_t$ is equivalent to $T$. Hence $T$ is rigid.
\end{proof}

\section{Applications}\label{sec-app}

\subsection{Reynolds operators}
As Reynolds operators on a Lie algebra $\mathfrak{g}$ are a particular case of $H$-twisted Rota-Baxter operators, we can apply the results of previous sections to Reynolds operators.

The following result is the Reynolds operator analogue of Theorem \ref{mc-thm}.

\begin{thm}
Let $(\mathfrak{g}, \mu = [~,~])$ be a Lie algebra. Then the graded vector space $\bigoplus_n \mathrm{Hom}(\wedge^n \mathfrak{g}, \mathfrak{g})$ carries an $L_\infty$-algebra structure with only nontrivial $l_2$ and $l_3$ that are given by (\ref{l_2}) and (\ref{l_3}) with $H = -\mu$. Moreover, a linear map $R : \mathfrak{g} \rightarrow \mathfrak{g}$ is a 
Reynolds operator on $\mathfrak{g}$ if and only if $R$ is a Maurer-Cartan element in the $L_\infty$-algebra.
\end{thm}

Given a Reynolds operator $R$ on the Lie algebra $\mathfrak{g}$, the cohomology induced from the corresponding Maurer-Cartan element is called the cohomology of $R$. This cohomology of $R$ can also be seen as the Chevalley-Eilenberg cohomology of a certain Lie algebra with coefficients in a representation. More precisely, if $R: \mathfrak{g} \rightarrow \mathfrak{g}$ is a Reynolds operator, then $\mathfrak{g}$ has a new Lie algebra structure with bracket $[~,~]_R$ given by (\ref{rey-new-brkt}).
The Lie algebra $(\mathfrak{g}, [~,~]_R)$ has a representation on $\mathfrak{g}$ itself by the following
\begin{align*}
 x ~\overline{\bullet}_R~ y := [Rx, y] + R ([y, x] - [y, Rx]), ~\text{ for } x, y \in \mathfrak{g}.
\end{align*}
Then the corresponding Chevalley-Eilenberg cohomology is isomorphic to the cohomology of $R$.

Deformations of a Reynolds operator can also be defined as of Section \ref{sec-4}. Such deformations are governed by the cohomology of the Reynolds operator.

\subsection{Twisted triangular {\bf r}-matrices}

In this subsection, we introduce cohomology of twisted triangular $r$-matrices from the points of view of twisted Rota-Baxter operators. Let $\mathfrak{g}$ be a Lie algebra. Consider the coadjoint representation $\mathfrak{g}^*$ of the Lie algebra $\mathfrak{g}$. By considering the coadjoint representation in Theorem \ref{mc-thm}, we get the following.

\begin{thm}
Let $\mathfrak{g}$ be a Lie algebra and $\psi \in \wedge^3 \mathfrak{g}^*$ be a $3$-cocycle of $\mathfrak{g}$ with coefficients in $\mathbb{K}$. Then the graded vector space $ \bigoplus_{n \geq 0} \mathrm{Hom} (\wedge^{n} \mathfrak{g}^*, \mathfrak{g})$ carries an $L_\infty$-algebra structure given by
\begin{align*}
l_1 = 0, \qquad l_2  =  \llbracket ~, ~ \rrbracket ,  \qquad
l_3 =   \llbracket ~, ~, ~ \rrbracket
\end{align*}
and trivial higher brackets. Here the binary bracket $\llbracket ~, ~ \rrbracket$ and the ternary bracket $\llbracket ~, ~, ~ \rrbracket$ as of (\ref{l_2}) and (\ref{ternary-t}), and we use $H = \psi^\sharp : \wedge^2 \mathfrak{g} \rightarrow \mathfrak{g}^*$ which is a $2$-cocycle of $\mathfrak{g}$ with coefficients in the coadjoint representation $\mathfrak{g}^*$.
An element ${\bf r } \in \wedge^2 \mathfrak{g}$ is a $\psi$-twisted triangular $r$-matrix if and only if ${\bf r}^\sharp : \mathfrak{g}^* \rightarrow \mathfrak{g}$ is a Maurer-Cartan element in the above $L_\infty$-algebra.
\end{thm}


It follows from the above theorem that a $\psi$-twisted triangular $r$-matrix ${\bf r} \in \wedge^2 \mathfrak{g}$ induces a differential 
\begin{align*}
d_{{\bf r}^\sharp} : \mathrm{Hom}(\wedge^n \mathfrak{g}^*, \mathfrak{g}) \rightarrow \mathrm{Hom}( \wedge^{n+1} \mathfrak{g}^*, 
\mathfrak{g}),~d_{{\bf r}^\sharp} := l_2 ({\bf  r}^\sharp, -) - \frac{1}{2}~ l_3( {\bf r}^\sharp, {\bf r}^\sharp, - ).
\end{align*}
The corresponding cohomology groups are called the cohomology of the $\psi$-twisted triangular $r$-matric ${\bf r}$. This cohomology is isomorphic to the Chevalley-Eilenberg cohomology of the Lie algebra $\mathfrak{g}^*$ given in (\ref{tw-bracket}) with coefficients in $\mathfrak{g}$ given by
\begin{align*}
\alpha ~\overline{\bullet}~ x = [{\bf r}^\sharp (\alpha), x] + {\bf r}^\sharp ( ad^*_x \alpha + \psi^\sharp (x, {\bf r}^\sharp (\alpha)),~ \text{ for } \alpha \in \mathfrak{g}^*,~ x \in \mathfrak{g}.
\end{align*}
When $\psi = 0$, we get $l_3 = 0$. Therefore, the corresponding cohomology groups are the standard cohomology of the classical triangular $r$-matrix \cite{tang}.

By the same spirit as of section \ref{sec-4}, we may study deformations of a $\psi$-twisted triangular $r$-matrix, their equivalences and rigidity.

\section{NS-Lie algebras}\label{sec-ns}
In this section, we introduce a new algebraic structure which is related to twisted Rota-Baxter operators in the same way pre-Lie algebras are related to Rota-Baxter operators. We call such algebras as NS-Lie algebras.
Further study on NS-Lie algebras is postponed to a forthcoming paper.

\begin{defn}
An {\bf NS-Lie algebra} is a vector space $L$ together with bilinear operations $\circ , \curlyvee : L \times L \rightarrow L$ in which $\curlyvee$ is skew-symmetric and satisfying the following two identities
\begin{itemize}
\item[(NS1)] $Ass (x, y, z) - Ass (y,x,z) + (x \curlyvee y) \circ z = 0$, \\
$\mathrm{ i.e. }~~ (x \circ y) \circ z  - x \circ (y \circ z ) - (y \circ x)\circ z + y \circ ( x \circ z ) + (x \curlyvee y) \circ z = 0,$
\item[(NS2)] $x \curlyvee (y * z) + c. p. + x \circ (y \curlyvee z) + c.p. = 0,$\\
$\mathrm{ i.e. }~~  x \curlyvee (y * z) + y \curlyvee (z * x) + z \curlyvee (x * y) +  x \circ (y \curlyvee z) + y \circ (z \curlyvee x) + z \circ (x \curlyvee y)  = 0,$
\end{itemize}
for all $x, y, z \in L$. Here $x * y = x \circ y - y \circ x + x \curlyvee y$, for $x, y \in L$.
\end{defn}

\begin{remark}
If the bilinear operation $\circ$ in the above definition is trivial, one gets that $(L, \curlyvee)$ is a Lie algebra. On the other hand, if $\curlyvee$ is trivial, then $(L, \circ)$ becomes a pre-Lie algebra. Thus, NS-Lie algebras are a generalization of both Lie algebras and pre-Lie algebras.
\end{remark}

Let us use the following notation in an NS-Lie algebra. For $x \in L$, we consider the map $l_x : L \rightarrow L$, $l_x (y) = x \circ y$, for $y \in L$. Then the identity (NS1) can be simply written as
\begin{align}\label{comm-ns}
[l_x, l_y] = l_{x * y},
\end{align}
where on the left-hand side, we use the commutator bracket of linear maps on $L$.

\begin{prop}
Let $(L, \circ, \curlyvee)$ be an NS-Lie algebra. Then $L$ equipped with the bracket 
\begin{align*}
[x, y] := x \circ y - y \circ x + x \curlyvee y
\end{align*}
is a Lie algebra (called the adjacent Lie algebra of $L$, denoted by $L_{\mathrm{Lie}}$). Morever, the Lie algebra $L_{\mathrm{Lie}}$ has a representation on $L$ given by $x \bullet m = x \circ m$, for $x \in L_{\mathrm{Lie}}$, $m \in L$.
\end{prop}

\begin{proof}
For any $x, y, z \in L$, we have
\begin{align}
[x,[y,z]] =~& [x, y \circ z - z \circ y + y \curlyvee z ] \nonumber \\
=~& x \circ (y \circ z) - x \circ (z \circ y) + x \circ (y \curlyvee z) - (y \circ z) \circ x + (z \circ y) \circ x  - (y \curlyvee z) \circ x + x \curlyvee (y *z). \label{jac-1}
\end{align}
Similarly,
\begin{align}
[y,[z,x]] =~& y \circ (z \circ x) - y \circ (x \circ z) + y \circ (z \curlyvee x) - (z \circ x) \circ y + (x \circ z) \circ y  - (z \curlyvee x) \circ y + y \curlyvee (z *x), \label{jac-2}\\
[z,[x,y]] =~& z \circ (x \circ y) - z \circ (y \circ x) + z \circ (x \curlyvee y) - (x \circ y) \circ z + (y \circ x) \circ z  - (x \curlyvee y) \circ z + z \curlyvee (x *y). \label{jac-3}
\end{align}
By adding (\ref{jac-1}), (\ref{jac-2}), (\ref{jac-3}) and using (NS1), we get
\begin{align*}
[x,[y,z]] + [y,[z,x]] + [z, [x,y]] =~& - \cancel{( z \curlyvee y) \circ x} - \bcancel{(x \curlyvee z) \circ y} - \xcancel{(y \curlyvee x) \circ z} \\&+ x \circ (y \curlyvee z) - \cancel{(y \curlyvee z) \circ x}  + y \circ (z \curlyvee x)  - \bcancel{(z \curlyvee x) \circ y} + z \circ (x \curlyvee y) - \xcancel{(x \curlyvee y) \circ z}  \\   & + x \curlyvee (y *z) + y \curlyvee (z *x) + z \curlyvee (x *y).
\end{align*}
This is $0$ by (NS2). Hence $(L, [~,~])$ is a Lie algebra.

The last part follows from the identity (\ref{comm-ns}).
\end{proof}

A morphism between NS-Lie algebras is a linear map between underlying vector spaces that commute with structure maps. We denote the category of NS-Lie algebras by ${\bf NS}.$ The construction of the above proposition yields a functor $(~)_c : {\bf NS} \rightarrow {\bf Lie}$ from the category of NS-Lie algebras to the category of Lie algebras.

\begin{prop}\label{prop-nij-ns}
Let $(\mathfrak{g}, [~,~])$ be a Lie algebra and $N : \mathfrak{g} \rightarrow \mathfrak{g}$ be a Nijenhuis operator. Then
\begin{align*}
x \circ y = [Nx, y]   \qquad \text{ and } \qquad x \curlyvee y = - N [x, y]
\end{align*}
defines an NS-Lie algebra structure on $\mathfrak{g}$.
\end{prop}

\begin{proof}
For any $x, y, z \in \mathfrak{g}$, we have
\begin{align*}
&(x \circ y) \circ z  - x \circ (y \circ z ) - (y \circ x)\circ z + y \circ ( x \circ z )\\
&= [N [Nx,y], z] - [Nx,[Ny,z]] - [N[Ny,x],z] + [Ny, [Nx,z]]\\
&= [N [Nx,y], z] - [[Nx, Ny], z] - \cancel{[Ny, [Nx,z]]} - [N[Ny,x],z] + \cancel{[Ny, [Nx,z]]}\\
&= \cancel{[N [Nx,y], z]} - \cancel{[N[Nx,y],z]} + \bcancel{[N[Ny,x],z]} + [N^2[x,y],z]  - \bcancel{[N[Ny,x],z]} \\
&= [N^2[x,y], z] = - [N(x \curlyvee y), z] = - (x \curlyvee y) \circ z.
\end{align*}
Hence we have proved (NS1). To prove (NS2), we first recall that for a Nijenhuis operator $N$, the deformed bracket
\begin{align*}
[x,y]_N := x \circ y - y \circ x + x \curlyvee y = [Nx, y] - [Ny, x] - N[x,y]
\end{align*}
is a Lie bracket on $\mathfrak{g}$ and this Lie bracket is compatible with the given Lie bracket $[~,~]$ on $\mathfrak{g}$ \cite{yks}. In other words, the sum $[~,~] + [~,~]_N$ is also a Lie bracket, equivalently,
\begin{align*}
[x, [y,z]_N] + c.p. + [x, [y,z]]_N + c.p. = 0.
\end{align*}
Thus, we have
\begin{align*}
x \curlyvee (y * z) + c.p. + x \circ ( y \curlyvee z) + c.p. =~& - N [x, [y,z]_N] + c.p. - [Nx, N[y,z]] + c.p. \\
=~& - N [x, [y,z]_N] + c.p. - N [x, [y,z]]_N + c.p. = 0.
\end{align*}
Hence the result is proved.
\end{proof}

In \cite{leroux} Leroux introduced NS-algebras as a generalization of dendriform algebras. We show that NS-algebras give rise to NS-Lie algebras.

\begin{defn}
An (associative) NS-algebra is a vector space $A$ together with three bilinear operations $\prec, \succ, \square : A \otimes A \rightarrow A$ satisfying the following identities
\begin{align}
( x \prec y ) \prec z = x \prec ( y \circledast z ), ~~~~ (x \succ y) \prec z = x \succ ( y \prec z), ~~~~ ( x \circledast y) \succ z = x \succ ( y \succ z), \\
( x ~\square~ y ) \prec z + ( x \circledast y) ~\square~ z = x \succ ( y ~\square ~z ) + x ~\square~ ( y \circledast z),~~~ \text{ for } x, y, z \in A.  \label{ns-ler-2}
\end{align}
Here $\circledast : A \otimes A \rightarrow A$ is the map defined by $x \circledast y = x \prec y + x \succ y + x ~\square~ y$, for $x, y \in A$.
\end{defn}

\begin{prop}
Let $(A, \prec, \succ, \square)$ be an (associative) NS-algebra. Then $(A, \circ, \curlyvee)$ is an NS-Lie algebra, where
\begin{align*}
x \circ y = x \succ y - y \prec x ~~~~ \text{ and } ~~~~ x \curlyvee y = x ~\square~ y - y ~ \square ~ x. 
\end{align*}
\end{prop}

\begin{proof}
By a direct calculation, we have
\begin{align*}
Ass (x, y, z) = ( x \circ y) \circ z - x \circ ( y \circ z) 
&= ( x \succ y) \succ z - ( y \prec x) \succ z - z \prec ( x \succ y) + z \prec ( y \prec x) \\
&- x \succ ( y \succ z) + x \succ ( z \prec y) + ( y \succ z ) \prec x - ( z \prec y) \prec x \\
&= - ( x \prec y) \succ z - ( x~ \square~ y) \succ z  - ( y \prec x) \succ z - z \prec ( x \succ y) \\
&+ x \succ (z \prec y) + ( y \succ z) \prec x - z \prec ( y \succ x) - z \prec ( y ~\square ~ x).
\end{align*}
Therefore,
\begin{align*}
Ass (x,y,z) - Ass (y, x, z) &= - ( x \prec y) \succ z - ( x~ \square~ y) \succ z  - ( y \prec x) \succ z - z \prec ( x \succ y) \\
& + x \succ (z \prec y) + ( y \succ z) \prec x - z \prec ( y \succ x) - z \prec ( y ~\square ~ x) \\
&+ ( y \prec x) \succ z + ( y~ \square~ x) \succ z  + ( x \prec y) \succ z + z \prec ( y \succ x) \\
& - y \succ (z \prec x) - ( x \succ z) \prec y + z \prec ( x \succ y) + z \prec ( x ~\square ~ y) \\
&= - ( x ~ \square~ y) \succ z - z \prec ( y ~\square~ x) + ( y ~\square~x ) \succ z + z \prec ( x ~\square~y) \\
&= - ( x \curlyvee y) \succ z + z \prec ( x \curlyvee y) = - ( x \curlyvee y) \circ z. 
\end{align*}
Thus we have proved (NS1). Next, we have from (\ref{ns-ler-2}) that
\begin{align*}
x~ \square~ ( y \circledast z ) - ( x \circledast y )~ \square~ z + x \succ ( y ~ \square ~z) - ( x ~ \square~ y) \prec z = 0,\\
y~ \square~ ( z \circledast x ) - ( y \circledast z )~ \square~ x + y \succ ( z ~ \square ~x) - ( y ~ \square~ z) \prec x = 0,\\
z~ \square~ ( x \circledast y ) - (z \circledast x )~ \square~ y + z \succ ( x ~ \square ~y) - ( z ~ \square~ x) \prec y = 0.
\end{align*}
By adding all these, we get
\begin{align}\label{onon}
 x \curlyvee ( y \circledast z ) + y \curlyvee ( z \circledast x ) + z \curlyvee ( x \circledast y ) + x \circ ( y ~\square~z) + y \circ ( z ~\square~ x) + z \circ ( x~\square~ y) =0.
\end{align}
Interchanging $y$ and $z$, we get
\begin{align}\label{toto}
x \curlyvee ( z \circledast y ) + z \curlyvee ( y \circledast x ) + y \curlyvee ( x \circledast z ) + x \circ ( z ~\square~y) + z \circ ( y ~\square~ x) + y \circ ( x~\square~ z) =0.
\end{align}
Finally, by subtracting (\ref{toto}) from (\ref{onon}), we get (NS2). This completes the proof.
\end{proof}


\subsection{Relation with twisted Rota-Baxter operators}

\begin{prop}\label{prop-tw-ns}
Let $T : M \rightarrow \mathfrak{g}$ be a $H$-twisted Rota-Baxter operator. Then
\begin{align*}
u \circ v  = T(u) \bullet v   \qquad \text{ and } \qquad  u \curlyvee v = H (Tu, Tv), ~ \text{ for } u, v \in M,
\end{align*}
defines an NS-Lie algebra structure on $M$.
\end{prop}

\begin{proof}
For any $u,v, w \in M,$
\begin{align*}
&(u \circ v ) \circ w - u \circ (v \circ w) - (v \circ u) \circ w + v \circ (u \circ w) \\
&= T(T(u) \bullet v) \bullet  w - T(u) \bullet (T(v) \bullet w) - T(T(v) \bullet u)\bullet w + T (v) \bullet (T(u) \bullet w) \\
&= T (T(u) \bullet v) \bullet w - [Tu,Tv] \bullet w - \cancel{T (v) \bullet (T(u) \bullet w)}  - T(T(v) \bullet u)\bullet w + \cancel{T (v) \bullet (T(u) \bullet w)}  \\
&= \cancel{ T (T(u) \bullet v) \bullet w } - \cancel{T(T(u) \bullet v) \bullet w} + \bcancel{T(T(v) \bullet u) \bullet w}  - TH (Tu,Tv) \bullet w - \bcancel{T(T(v) \bullet u)\bullet w} \\
&= - TH (Tu,Tv)\bullet w = - (u \curlyvee v) \circ w.
\end{align*}
Thus the condition (NS1) holds. For (NS2), we observe that
\begin{align*}
&u \curlyvee (v * w ) + c.p. + u \circ (v \curlyvee w) + c.p.\\
&= H (Tu, T (v * w)) + c.p. + T(u) \bullet H (Tv, Tw) + c.p.\\
&= H (Tu, [Tv, Tw]) + c.p. + T(u) \bullet H(Tv, Tw) + c.p. = 0 ~~~~ (\text{as } H \text{ is a } 2\text{-cocycle}).
\end{align*}
Hence $(M, \circ , \curlyvee)$ is an NS-Lie algebra.
\end{proof}

\begin{remark}
Let $N : \mathfrak{g} \rightarrow \mathfrak{g}$ be a Nijenhuis operator on a Lie algebra $(\mathfrak{g}, [~,~])$. It follows from Example \ref{exam-n-tw} and Proposition \ref{prop-tw-ns} that the vector space $\mathfrak{g}$ carries an NS-Lie algebra structure $(\mathfrak{g}, \diamond, \curlyvee)$ given by
\begin{align*}
x \diamond y = \mathrm{id}(x) \bullet y = x \bullet y = [Nx, y] ~~~ \text{ and } ~~~ x \curlyvee y = H (\mathrm{id}(x), \mathrm{id}(y)) = H (x, y) = - N [x, y].
\end{align*}
This NS-Lie algebra structure on $\mathfrak{g}$ is precisely the one obtained in Proposition \ref{prop-nij-ns}.
\end{remark}


Let $T: M \rightarrow \mathfrak{g}$ be a $H$-twisted Rota-Baxter operator. Suppose $\mathfrak{g}'$ is another Lie algebra, $M'$ is a $\mathfrak{g}'$-module and $H' \in C^2_{\mathrm{CE}}(\mathfrak{g}', M')$ is a $2$-cocycle. Let $T' : M' \rightarrow \mathfrak{g}'$ be a $H'$-twisted Rota-Baxter operator. A weak morphism from $T$ to $T'$ consist of linear maps $\phi : \mathfrak{g} \rightarrow \mathfrak{g}'$ and $\psi : M \rightarrow M'$ satisfying $\psi ( x \bullet u) = \phi(x) \bullet \psi (u)$, $\psi  \circ H  = H' \circ ( \phi \otimes \phi)$ and $\phi \circ T = T' \circ \psi$, for $x\in \mathfrak{g}$ and $u \in M$. Thus, a morphism of twisted Rota-Baxter operators in the sense of Definition \ref{rota-morphism} is a weak morphism $(\phi, \psi)$ in which $\phi$ is a Lie algebra morphism.

Let $\bf{ TRB}$ denotes the category of twisted Rota-Baxter operators and weak morphisms between them. In the previous proposition, we construct an NS-Lie algebra from a twisted Rota-Baxter operator. In fact, this construction gives rise to a functor $F : {\bf TRB} \rightarrow {\bf NS}$.

Let $(L, \circ, \curlyvee)$ be an NS-Lie algebra with the adjacent Lie algebra $L_{\mathrm{Lie}} = (L, [~,~])$. Then it has been shown that the Lie algebra $L_{\mathrm{Lie}}$ has a (well-defined) representation on $L$ given by
\begin{align*}
x \bullet u := x \circ u,~ \text{ for } x \in L_{\mathrm{Lie}},~ u \in L.
\end{align*}
With this notation, the map $H : L_{\mathrm{Lie}} \times L_{\mathrm{Lie}} \rightarrow L$, $H(x,y) = x \curlyvee y$ is a $2$-cocycle in the Chevalley-Eilenberg cohomology of the Lie algebra $L_{\mathrm{Lie}}$ with coefficients in the representation $L$. It is easy to see that the condition (NS2) is equivalent to the fact that $H$ is a $2$-cocycle. The identity map $1 : L \rightarrow L_{\mathrm{Lie}}$ is then a $H$-twisted Rota-Baxter operator. This construction of a twisted Rota-Baxter operator from an NS-Lie algebra is also functorial and defines a  functor $G : {\bf NS} \rightarrow {\bf TRB}$. Moreover, we get an adjoint relation
\begin{align*}
\mathrm{Hom}_{\bf NS} (L, F (T)) \cong \mathrm{Hom}_{\bf TRB}(G(L), T),
\end{align*}
where $T$ is a $H$-twisted Rota-Baxter operator and $L$ is any NS-Lie algebra.

\section{Twisted generalized complex structures}\label{sec-tgcs}

In this section, we introduce twisted generalized complex structures on modules over a Lie algebra. When one considers the coadjoint module $\mathfrak{g}^*$ of a Lie algebra $\mathfrak{g}$, one gets twisted generalized complex structures on $\mathfrak{g}$.

Let $\mathfrak{g}$ be a Lie algebra, $M$ a $\mathfrak{g}$-module and $H \in C^2_{\mathrm{CE}} (\mathfrak{g}, M)$ be a $2$-cocycle. Consider the semi-direct product Lie algebra $\mathfrak{g} \ltimes_H M$ with the bracket given in (\ref{semi-dir-brkt}). Any linear map $J: \mathfrak{g} \oplus M \rightarrow \mathfrak{g} \oplus M$ must be of the form
\begin{align}\label{gcs-o}
J = \left( \begin{array}{cc}
N    &      T\\
\sigma     &    -S
\end{array}
\right),
\end{align}
for some linear maps $N \in \mathrm{End} (\mathfrak{g}), S \in \mathrm{End} (M),~ T : M \rightarrow \mathfrak{g}$ and $\sigma : \mathfrak{g} \rightarrow M$. These maps are called the structure components of $J$. The reason behind considering the linear map as $-S$ (instead of $S$) will be clear from Proposition \ref{gcs-o-gcs-l}.

\begin{defn}\label{defn-gcs-o}
An $H$-twisted generalized complex structure on $M$ over the Lie algebra $\mathfrak{g}$ is a linear map $J : \mathfrak{g} \oplus M \rightarrow \mathfrak{g} \oplus M$ that satisfies the following conditions
\begin{itemize}
\item[(i)] $J$ is almost complex: ~ $J^2 = - \mathrm{id}$,
\item[(ii)] integrability condition: ~
$[Jr, Js] - [r,s] - J ([Jr, s]+ [r, Js]) = 0, ~ \text{for } r, s \in \mathfrak{g} \oplus M.$
\end{itemize}
\end{defn}

Note that we have used the $H$-twisted semi-direct product bracket (\ref{semi-dir-brkt}) in the above integrability condition. A generalized complex structure on $M$ over the Lie algebra $\mathfrak{g}$ is a $H$-twisted generalized complex structure for $H = 0$.

In \cite{crainic} Crainic gives a characterization of generalized complex manifolds. A similar theorem in the context of $H$-twisted generalized complex structures is the following.

\begin{thm}\label{gcs-o-char}
A linear map $J :  \mathfrak{g} \oplus M \rightarrow \mathfrak{g} \oplus M$ of the form (\ref{gcs-o}) is a $H$-twisted generalized complex structure on $M$ over the Lie algebra $\mathfrak{g}$ if and only if its structure components satisfy the following identities:
\begin{align}
&NT = TS,\label{gcs-equiv-1}\\
&N^2 + T \sigma = -\mathrm{id},\\
&S \sigma = \sigma N, \\
&S^2 + \sigma T = - \mathrm{id}, \label{gcs-equiv-4}\\
&[Tu, Tv] = T (Tu \bullet v - Tv \bullet u), \label{gcs-equiv-5}\\
&Tu \bullet Sv - Tv \bullet Su - H (Tu, Tv) = S (Tu \bullet v - Tv \bullet u), \label{gcs-equiv-6}\\
&[Nx, Tu] - N [x, Tu] = T ( Nx \bullet u - x \bullet Su + H (x, Tu)), \label{gcs-equiv-7}\\
&\sigma [Tu, x] - Tu \bullet \sigma x - H(Tu, Nx) = x \bullet u + Nx \bullet Su - S ( Nx \bullet u - x \bullet Su + H (x, Tu)), \label{gcs-equiv-8}\\
&[Nx, Ny] - [x, y] - N ([Nx, y]+ [x, Ny]) = T ( x \bullet \sigma y - y \bullet \sigma x  + H (x, Ny) - H (y,Nx)), \label{gcs-equiv-9}\\
&Nx \bullet \sigma y - Ny \bullet \sigma x + H (Nx, Ny) - H (x,y) - \sigma ([Nx, y]+ [x, Ny])= \label{gcs-equiv-10} \\& \qquad \qquad \qquad - S (x \bullet \sigma y - y \bullet \sigma x + H (x, Ny) - H (y,Nx)), \nonumber
\end{align}
for all $x, y \in \mathfrak{g}$ and $u, v \in M$.
\end{thm}

\begin{proof}
The condition $J^2 = - \mathrm{id}$ is same as
\begin{align*}
\left( \begin{array}{c}
N^2(x) + NT (u) + T \sigma (x) - TS (u)  \\
\sigma N (x) + \sigma T(u) - S \sigma (x) + S^2 (u)    
\end{array}
\right) = \left( \begin{array}{c}
-x  \\
-u   
\end{array}
\right),
\end{align*}
which is equivalent to the identities (\ref{gcs-equiv-1})-(\ref{gcs-equiv-4}). Next consider the integrability condition of  $J$. For $r=u, ~ s = v \in M$, the integrability criteria is equivalent to (\ref{gcs-equiv-5}) and (\ref{gcs-equiv-6}). For $r = x \in \mathfrak{g}$ and $s = u \in M$, the integrability is equivalent to (\ref{gcs-equiv-7}) and (\ref{gcs-equiv-8}). Finally, for $r = x ,~ s = y \in \mathfrak{g}$, we get the identities (\ref{gcs-equiv-9}) and (\ref{gcs-equiv-10}).
\end{proof}

\begin{remark}
Note that the condition (\ref{gcs-equiv-5}) implies that the map $T: M \rightarrow \mathfrak{g}$ is a Rota-Baxter operator. Moreover, the condition (\ref{gcs-equiv-6}) is equivalent to the fact that
\begin{align*}
\mathrm{Gr}((T,S)) = \{ ( Tu, Su)|~ u \in M \} \subset \mathfrak{g} \oplus M
\end{align*}
is a subalgebra of the $(-H)$-twisted semi-direct product $\mathfrak{g} \ltimes_{-H} M$.
\end{remark}

The above theorem ensures the following examples of (twisted) generalized complex structures on modules over Lie algebras.

\begin{exam} (Opposite g.c.s.) Let $J = \left( \begin{array}{cc}
N   &      T\\
\sigma     &   -S
\end{array}
\right)$ be a $H$-twisted generalized complex structure on $M$ over $\mathfrak{g}$. Then $\overline{J} = \left( \begin{array}{cc}
N   &      - T\\
- \sigma     &   -S
\end{array}
\right)$ is a $(-H)$-twisted generalized complex structure, called the opposite of $J$.
\end{exam}

\begin{exam}
Let $T: M \rightarrow \mathfrak{g}$ be an invertible Rota-Baxter operator. Then $J = \left( \begin{array}{cc}
0   &      T\\
- T^{-1}     &   0
\end{array}
\right)$ is a generalized complex structure on $M$ over $\mathfrak{g}$ ($H = 0$).
\end{exam}

Complex structures on modules over Lie algebra also lie in generalized complex structures. We start with the following known definition.

\begin{defn}
A complex structure on a Lie algebra $\mathfrak{g}$ is a linear map $I : \mathfrak{g} \rightarrow \mathfrak{g}$ satisfying $I^2 = - \mathrm{id}$ and
\begin{align*}
[Ix, Iy] - [x, y] - I ( [Ix, y]+[x, Iy]) = 0, ~ \text{for } x, y \in \mathfrak{g}.
\end{align*}
\end{defn}

\begin{defn}
Let $\mathfrak{g}$ be a Lie algebra and $M$ be a $\mathfrak{g}$-module. A complex structure on $M$ over the Lie algebra $\mathfrak{g}$ is a pair $(I, I_M)$ of linear maps $I \in \mathrm{End}(\mathfrak{g})$ and $I_M \in \mathrm{End}(M)$ satisfying the followings
\begin{itemize}
\item[$\triangleright$] $I$ is a complex structure on $\mathfrak{g}$,
\item[$\triangleright$] $I_M^2 = - \mathrm{id}$ and  \begin{equation}\label{iden-s}
I(x) \bullet I_M (u) - x \bullet u - I_M ( I(x) \bullet u + x \bullet I_M (u)) = 0, \text{ for }  x \in \mathfrak{g}, u \in M.
\end{equation}
\end{itemize}
\end{defn}

\begin{prop}
A pair $(I, I_M)$ is a complex structure on $M$ over the Lie algebra $\mathfrak{g}$ if and only if $I \oplus I_M : \mathfrak{g} \oplus M \rightarrow \mathfrak{g} \oplus M$ is a complex structure on the semi-direct product Lie algebra $\mathfrak{g} \ltimes M$.
\end{prop}

\begin{proof}
The linear map $I \oplus I_M$ is a complex structure on the semi-direct product $\mathfrak{g} \ltimes M$ if and only if $(I \oplus I_M)^2 = - \mathrm{id}$ (equivalently, $I^2 = -\mathrm{id}$ and $I_M^2 = - \mathrm{id})$ and
\begin{align*}
[ (Ix, I_M u), (Iy, I_M v)] - [(x,u), (y, v)] - (I \oplus I_M ) ( [ (Ix, I_M u), (y, v)] + [(x, u), (Iy, I_M v)] ) = 0,
\end{align*}
or, equivalently, 
\begin{align*}
\big( [Ix, Iy] - [x, y] - I ([Ix, y] + [x, Iy]), ~& Ix \bullet I_M v - Iy \bullet I_M v - (x \bullet v - y \bullet u) \\ &- I_M (Ix \bullet v - y \bullet I_M u + x \bullet I_M v - Iy \bullet u) \big) = 0.
\end{align*}
The last identity is equivalent to the fact that $I$ is a complex structure on $\mathfrak{g}$ and the identity (\ref{iden-s}) holds. In other words $(I, I_M)$ is a complex structure on $M$ over the Lie algebra $\mathfrak{g}$.
\end{proof}

\begin{exam}
Let $(I, I_M)$ be a complex structure on $M$ over the Lie algebra $\mathfrak{g}$. Then $J = \left( \begin{array}{cc}
I   &      0\\
0     &   I_M
\end{array}
\right)$ is a generalized complex structure on $M$ over $\mathfrak{g}$. Note that here $S= - I_M.$
\end{exam}


Next, we consider twisted generalized complex structures on Lie algebras and show that they are related to twisted generalized complex structures on coadjoint modules.
Let $\mathfrak{g}$ be a Lie algebra and $\psi \in \wedge^3 \mathfrak{g}^*$ be a $3$-cocycle in the cohomology of $\mathfrak{g}$ with coefficients in $\mathbb{K}$.  Consider the coadjoint representation of $\mathfrak{g}$ on $\mathfrak{g}^*$. Then $\psi^\sharp : \wedge^2 \mathfrak{g} \rightarrow \mathfrak{g}^*$ is a $2$-cocycle in the cohomology of $\mathfrak{g}$ with coefficients in $\mathfrak{g}^*$. The corresponding $\psi^\sharp$-twisted semi-direct product algebra $\mathfrak{g} \ltimes_{\psi^\sharp} \mathfrak{g}^*$ is given by the bracket
\begin{align*}
[(x, \alpha), (y, \beta)] = ([x, y], \mathrm{ad}^*_x \beta - \mathrm{ad}^*_y \alpha + \psi^\sharp (x, y)).
\end{align*}
The direct sum $\mathfrak{g} \oplus \mathfrak{g}^*$ also carries a non-degenerate inner product given by
$\langle (x, \alpha), (y, \beta) \rangle = \frac{1}{2} ( \alpha (y) + \beta (x)).$
\begin{defn}
A $\psi$-twisted generalized complex structure on $\mathfrak{g}$ consists of a linear map $J : \mathfrak{g} \oplus \mathfrak{g}^* \rightarrow \mathfrak{g} \oplus \mathfrak{g}^*$ satisfying
\begin{itemize}
\item[(0)] orthogonality: ~~ $\langle Jr, Js \rangle = \langle r, s \rangle,$
\item[(i)] almost complex: ~~ $J^2 = - \mathrm{id}$,
\item[(ii)] integrability: ~~
$[Jr, Js ] - [r, s] - J ( [Jr, s] + [r, Js]) = 0, ~\text{for } r,s \in \mathfrak{g} \oplus \mathfrak{g}^*.$
\end{itemize}
\end{defn}

Note that the orthogonality condition (0) implies that $J$ must be of the form 
\begin{align}\label{gcs-l}
J = \left( \begin{array}{cc}
N    &      {\bf r}^\sharp \\
\sigma_\flat    &    - N^*
\end{array}
\right),
\end{align}
for some $N \in \mathrm{End}(\mathfrak{g}), ~ {\bf r} \in \wedge^2 \mathfrak{g}$ and $\sigma \in \wedge^2 \mathfrak{g}^*$. However, the conditions (i) and (ii) of the definition imposes some relations between structure components of $J$.

Thus, a $\psi$-twisted generalized complex structure on $\mathfrak{g}$ can also be considered as a triple $(N, {\bf r}, \sigma)$ such that the linear map $J$ of the form (\ref{gcs-l}) is almost complex and satisfies the integrability condition.



\begin{prop}\label{gcs-o-gcs-l}
Let $\mathfrak{g}$ be a Lie algebra and $\psi \in \wedge^3 \mathfrak{g}^*$ be a $3$-cocycle in the cohomology of $\mathfrak{g}$ with coefficients in $\mathbb{K}$. A triple $(N, {\bf r}, \sigma)$  is a $\psi$-twisted generalized complex structure on $\mathfrak{g}$ if and only if the linear map $J = \left( \begin{array}{cc}
N    &      {\bf r}^\sharp \\
\sigma_\flat    &    - N^*
\end{array}
\right) : \mathfrak{g} \oplus \mathfrak{g}^* \rightarrow \mathfrak{g} \oplus \mathfrak{g}^*$ is a $\psi^\sharp$-twisted generalized complex structure on the coadjoint module $\mathfrak{g}^*$ over the Lie algebra $\mathfrak{g}$.
\end{prop}






\vspace*{1cm}

\noindent {\bf Acknowledgements.} The author thanks the referee for his/her valuable advice which improved the paper. The research is supported by the fellowship of Indian Institute of Technology (IIT) Kanpur.

\vspace*{1cm}

\noindent {\bf Data Availability Statement.} Data sharing is not applicable to this article as no new data were created or analyzed in this study.

\end{document}